\documentclass[12pt]{amsart}
\usepackage{amssymb}
\usepackage{amsmath}
\usepackage{xcolor}
\usepackage{amscd}
\usepackage{pstricks}
\usepackage{pst-node}

\usepackage{geometry}
\newtheorem{thm}{Theorem}[section]
\newtheorem{cor}[thm]{Corollary}
\newtheorem{lem}[thm]{Lemma}
\newtheorem{prop}[thm]{Proposition}

\theoremstyle{definition}

\newtheorem{defn}[thm]{Definition}

\usepackage{graphics}
\usepackage{graphicx}

\newtheorem{rem}[thm]{Remark}

\numberwithin{equation}{section}

\newcommand{\N}{\mathbb{N}}
\newcommand{\R}{\mathbb{R}}
\newcommand{\C}{\mathbb{C}}
\newcommand{\Z}{\mathbb{Z}}

\newcommand{\trans}[1]{{}^t\kern-.2em{#1}}
\newcommand{\ytrans}[1]{{}^t\kern-.11em{#1}}
\newcommand{\Trans}[1]{{}^T\kern-.2em{#1}}
\newcommand{\lsup}[2]{{}^{#1}\kern-.1em{#2}}

\newcommand{\Ker}{\operatorname{Ker}}

\newcommand{\M}{\mathbf{\M}}

\renewcommand{\tilde}[1]{\widetilde{#1}}
\renewcommand{\hat}[1]{\widehat{#1}}
\DeclareFixedFont{\bgn}{OT1}{cmr}{m}{n}{20.74}
\DeclareFixedFont{\bgi}{OT1}{cmr}{m}{it}{20.74}

\newcommand{\bigzerou}{\smash{\lower1.7ex\hbox{\bgi O}}}

\makeatletter
\def\eqnarray{%
   \stepcounter{equation}%
   \def\@currentlabel{\p@equation\theequation}%
   \global\@eqnswtrue
   \m@th
   \global\@eqcnt\z@
   \tabskip\@centering
   \let\\\@eqncr
   $$\everycr{}\halign to\displaywidth\bgroup
       \hskip\@centering$\displaystyle\tabskip\z@skip{##}$\@eqnsel
      &\global\@eqcnt\@ne \hfil$\displaystyle{{}##{}}$\hfil
      &\global\@eqcnt\tw@ $\displaystyle{##}$\hfil\tabskip\@centering
      &\global\@eqcnt\thr@@ \hb@xt@\z@\bgroup\hss##\egroup
        \tabskip\z@skip
      \cr}
\makeatother

\makeatletter
\def\varin{\mathrel{\mathpalette\@varin\relax}}
\def\@varin#1{%
   \hbox{\setbox\z@\hbox{\m@th$#1\cup$}%
       \def\reserved@a{bold}%
       \dimen@\ifx\reserved@a\math@version .3\else .2\fi\p@
       \kern.5\wd\z@\kern-\dimen@
       \vrule\@width2\dimen@\@height1.08\ht\z@\@depth\z@
       \kern-\dimen@\kern-.5\wd\z@
       \box\z@}}
\makeatother
\begingroup
    \divide\count0 by 60
    \multiply\count0 by 60
    \advance\count1 by -\count0
    \ifnum\count2>11
         \ifnum\count2>12 \advance\count2 by -12\fi
         \def\ampm{\,pm}%
    \else
         \ifnum\count2=0 \advance\count2 by 12\fi
         \def\ampm{\,am}%
    \fi
    \xdef\daytime{%
         \ifnum\count2<10 0\fi \the\count2:%
         \ifnum\count1<10 0\fi \the\count1
         \ampm
    }%
    \xdef\Daytime{
         \the\count2:
         \ifnum\count1<10 0\fi \the\count1   
         \ampm                               
    }
\endgroup

\date{\today}
\begin{document}

\title[Sub-Laplacians and their spectrum of pseudo $H$-type nilmanifolds]{Spectral theory of a class of nilmanifolds attached to Clifford modules}

\author{Wolfram Bauer, Kenro Furutani, Chisato Iwasaki \\ and Abdellah Laaroussi}
\thanks{{\bf Corresponding author}: {\sc Wolfram Bauer}, Institut f\"{u}r Analysis, Welfengarten 1, 30167 Hannover, Germany. email: \texttt{bauer@math.uni-hannover.de}\endgraf 
{\sc Kenro Furutani}, Department of Mathematics, Tokyo University of Science, 2641 Yamazaki, Noda, Chiba 287-8510, Japan. email: 
\texttt{furutani$\_$kenro@ma.noda.tus.ac.jp}\endgraf
{\sc Chisato Iwasaki}, Department of Mathematics, School of Science, 
University of Hyogo,  2167 Shosha Himeji, Hyogo 671-2201, Japan. email: \texttt{iwasaki@sci.u-hyogo.ac.jp} \endgraf
{\sc Abdellah Laaroussi}, Institut f\"{u}r Analysis, Welfengarten 1, 30167 Hannover, Germany. email: \texttt{abdellah.laaroussi@math.uni-hannover.de}}

\thanks{The first and the last named author has been supported by the priority program SPP 2026 {\it geometry at infinity} of Deutsche Forschungsgemeinschaft (project number BA 3793/6-1), the second named author was supported by the Grant-in-aid for Scientific Research (C) 
No.17K05284, JSPS;  the third named author was supported by the Grant-in-aid for Scientific Research (C) No.24540189, JSPS} 




\begin{abstract}
We determine the spectrum of the sub-Laplacian on pseudo ${H}$-type nilmanifolds and present pairs of isospectral but non-diffeomorphic nilmanifolds with respect to the sub-Laplacian. We observe that these pairs are also isospectral with respect to the Laplacian. More generally, 
our method allows us to construct an arbitrary number of isospectral but mutually non-diffeomorphic nilmanifolds. Finally, we present two nilmanifolds of different dimensions such that the short time heat trace expansions of the corresponding sub-Laplace operators coincide up to a term which vanishes to infinite order as time tends to zero. 
\end{abstract}
\maketitle
\thispagestyle{empty}
\section{Introduction}
\label{Section_Introduction}
In 1966 Mark Kac's famous paper \cite{Kac} 
asked the questions {\it ''Can one hear the shape of a drum?''}.  
This work can be regarded as the beginning of a central topic of spectral geometry although the problem itself traces back to Hermann Weyl' s 
work at the beginning 
of the 20th century. Especially in the multi-dimensional situation, a negative answer to the above question was expected early on. 
Therefore, an important task was to construct isospectral but non-isometric  or even non-diffeomorphic manifolds. Such examples allow to determine geometric 
properties that are not determined by the spectrum. In high dimensions the first example of such manifolds 
was given by J. Milnor even earlier in 1964.  In \cite{Mi}  a pair of 16-dimensional flat tori have been constructed which are isospectral but non-isometric.  
Nowadays, a general  construction method by T. Sunada \cite{Su} and a wide range of examples are known, cf. \cite{GWW,G, Ik1, Ik2, Mc}. In particular, 
they include lens spaces, spherical space forms or Heisenberg manifolds.  
Generalizing the last example the isospectrality problem may be considered for quotients $\Gamma \backslash \mathbb{G}$ of nilpotent Lie groups $\mathbb{G}$ 
of step $k\geq 2$ by a lattice $\Gamma$. In the following we will call such manifolds {\it $k$-step nilmanifolds}. Via an adaptation of representation theoretical methods due to 
C.S. Gordon and N.E. Wilson, pairs of isospectral nilmanifolds $(\Gamma_1\backslash \mathbb{G},\,\, \Gamma_2\backslash \mathbb{G})$ of step $k \geq 3$ 
were constructed in \cite{G}. Different from the known examples based on Sunada's theorem these manifolds need not to be isospectral for the Laplacian on $1$-forms. 
 
In the realm of Riemannian geometry it remains an interesting problem to construct isospectral but non-diffeomorphic manifolds in a systematic way. Moreover, by  restricting M. Kac's question to specific sub-classes of smooth manifolds (e.g. spheres or certain nilmanifolds) or by considering the spectrum of geometric operators different from the Laplacian one is 
led to new classification problems.   
\vspace{1mm}\par
In the present paper we consider M. Kac's question for a class of subriemannian manifolds $M$ carrying a geometrically defined second order sub-elliptic differential operator, 
called sub-Laplacian. More precisely, in our setup $M$ is assumed to be a  nilmanifold of step 2 whose covering simply connected nilpotent Lie group is of {\it pseudo $H$-type}. 
Such groups are generalizations of the well known {\it Heisenberg type groups} introduced by A. Kaplan in \cite{Ka1}. Their Lie algebras are called of {\it pseudo $H$-type} as well and 
were first considered in \cite{Ci}. Pseudo $H$-type Lie algebras are constructed from Clifford algebras $C\ell_{r,s}$ of signature $(r,s)$ and their {\it (admissible) modules}, cf. Section 
\ref{Section_Pseudo_H_type_algebras} or \cite{Ci,FM2, FM1, FM} for a definition and more details. We also recall that the existence of lattices $\Gamma$ in pseudo $H$-type Lie groups $\mathbb{G}$ 
has been proven in \cite{FM}. With respect to a standard (integral) lattice we can therefore consider compact  left-coset spaces $\Gamma \backslash \mathbb{G}$. 
\vspace{1mm}\par 
Based on an explicit heat trace  formula for the sub-Laplacian combined with the recent classification of pseudo $H$-type algebras in \cite{FM2,FM1} we can give the negative 
answer to M. Kac's question in this non-standard setting and present a list of new examples. The subriemannian structure we deal with naturally extends  
to a Riemannian structure and we may as well consider the corresponding Laplacian $\Delta$ on $M$. As it turns out in our examples the difference 
$D:=\Delta-\Delta_{\textup{sub}}$ is a "sum-of-squares-operator", i.e. it can be expressed in the form 
\[
D={-\frac{1}{2}\sum_k {Z_k}^2}
\]
with globally defined vector fields $Z_k$ on $M$. Moreover, the operators $D$ and $\Delta_{\textup{sub}}$ commute and therefore $\Delta$ and 
$\Delta_{\textup{sub}}$ commute as well. As a consequence we have obtained new examples of isospectral, non-diffeomorphic manifolds in the usual sense, i.e. with respect to the Riemannian structure.    

Different from previously  studied isospectral, non-isometric quotients $(\Gamma _1\backslash \mathbb{G}, \Gamma _2\backslash \mathbb{G})$ for which the covering simply connected Lie group $\mathbb{G}$ is fixed and the lattice $\Gamma$ varies we note that in our construction also the group $\mathbb{G}$ is varying in the definition of both manifolds. Furthermore, by choosing the space dimension suitably high our method not only allows us to select pairs but any given number of isospectral, non-diffeomorphic nilmanifolds.
\vspace{1mm}\par 
Before stating the results more in detail we review some definitions. By a subriemannian manifold we understand a triple $(M, \mathcal{H}, \langle \cdot, \cdot \rangle)$ where $M$ is a smooth manifold (orientable and without boundary),  $\mathcal{H}$ is a bracket generating subbundle in the tangent bundle $TM$ and $\langle \cdot, \cdot \rangle$ denotes a family of inner products 
on $\mathcal{H}$ which smoothly vary with the base point. Recall that $\mathcal{H}$ is called {\it bracket generating} if vector fields taking values in $\mathcal{H}$ together with a finite number of 
their iterated brackets span the tangent space at any point of $M$. In addition, let us assume that $\langle  \cdot,\cdot\rangle$ is obtained by restricting a Riemannian metric on $M$  to $\mathcal{H}$. Then we can identify tangent and cotangent space via this metric and consider the orthogonal complement 
${N_{\mathcal{H}}}^{\perp}$ of the bundle 
$$N_{\mathcal{H}}=\big{\{}\theta\in T^{*}M~|~ \theta(X)=0,~{^{\forall}X}\in\mathcal{H}~\big{\}}$$ 
normal to $\mathcal{H}$. With the projection  $\mathcal{P}:T^*M\to {N_{\mathcal{H}}}^{\perp}$ along the normal bundle the {\it sub-Laplacian} $\Delta_{\textup{sub}}$ is defined as the composition 
$$\Delta_{\textup{sub}}= d^{*}\circ \mathcal{P}\circ d :C^{\infty}(M)\to C^{\infty}(M).$$
Here $d$ and $d^*$ denote the exterior derivative and its adjoint with respect to the Riemannian metric. 
Locally and modulo a first order operator $\Delta_{\textup{sub}}$ is expressed as a sum of squares of mutually orthogonal vector fields taking values in the subbundle $\mathcal{H}$. 
More generally and based on the construction of the {\it Popp measure}, an intrinsic sub-Laplacian can be defined in a different way on any equi-regular subriemannian manifold 
even if the family of inner products $\langle \cdot, \cdot \rangle$ is only given on $\mathcal{H}$, cf. \cite{ABGR}. 
The manifolds $M$ in this paper are equipped with a bracket generating subbundle $\mathcal{H}$ which is trivial as a vector bundle and a metric on $\mathcal{H}$ which naturally extends to  a 
Riemannian metric on $M$. Hence  it can be shown that in our setting both definitions  of the sub-Laplacian coincide. More precisely,  there is a globally defined frame $\{X_i\}$ of $\mathcal{H}$ which is 
orthonormal at any point and skew-symmetric with respect to a naturally chosen volume form such that: 
\footnote{In oder to simplify the heat kernel expression we have chosen the factor $\frac{1}{2}$ in front of the sum.}
\[
\Delta_{\textup{sub}}=-\frac{1}{2}\sum_i \,{X_i}^2.
\]
\par 
\vspace{1mm}\par 
As is well-known the {\it bracket generating property} (also called {\it H\"{o}rmander condition}) implies that $\Delta_{\textup{sub}}$ is a {\it sub-elliptic} operator 
(i.e. it satisfies an ``a priori estimate with a loss of derivative'', cf. \cite{Ho}). Clearly, this property does not depend on the chosen Riemannian metric. As a 
consequence it can be shown that the sub-Laplacian on a closed manifold $M$ has a compact resolvent and spectrum only consisting of eigenvalues with finite multiplicities (see \cite{Ho}). 
Hence we can define the notion of isospectrality of two given subriemannian manifolds by replacing the spectrum of the Laplacian with the spectrum of the sub-Laplacian. 
Unlike in the elliptic case a precise asymptotic expansion formula for the eigenvalue sequence for a general subriemannian manifold is unknown (however, see \cite{II}). 
In the special case where $M$ is a 2-step nilmanifolds it was shown in \cite{BFI1} that the heat trace of the sub-Laplacian admits an asymptotic expansion similar to the heat 
trace expansion of the Laplacian on a torus. 
\vspace{1mm}\par  
In order to detect isospectral (subriemannian) nilmanifolds we first need to determine the spectrum of the sub-Laplacian $\Delta_{\textup{sub}}^{\Gamma\backslash {\mathbb{G}}}$ on a 2-step nilmanifold 
$M= \Gamma \backslash \mathbb{G}$.  Based on an explicit expression of the heat kernel for $\Delta_{\textup{sub}}$ on the covering group $\mathbb{G}$ in \cite{BGG2,CCFI,Fu} 
a formula for the heat trace of $\Delta_{\textup{sub}}^{\Gamma\backslash {\mathbb{G}}}$ descended from $\mathbb{G}$ to $M= \Gamma \backslash \mathbb{G}$ was obtained in \cite{BFI2}. 
In case of a pseudo $H$-type group $\mathbb{G}$ this trace formula simplifies further and in principle can be used to explicitly calculate the spectrum of $\Delta_{\textup{sub}}$ on $M$. 
However, we need not to perform the full calculation. In order to identify isospectral manifolds it is sufficient to compare the corresponding trace formulas. 
\vspace{1ex}\par 
In a second step we need to classify non-diffeomorphic nilmanifolds $\Gamma_1 \backslash \mathbb{G}_1$ and $\Gamma_2 \backslash \mathbb{G}_2$ of the same dimension. First, we 
reduce this task to a classification of pseudo $H$-type Lie algebras up to isomorphisms (cf. Corollary \ref{corollary_non_diffeomorphic_nilmanifolds}). Then we apply the  
very recent classification results in \cite{FM1,FM2}. 
\vspace{1ex}\par 
The paper is organized as follows: In Section \ref{Section_Heat_kernel_Liegroups} we introduce the sub-Laplacian on a general 2-step nilpotent Lie group $\mathbb{G}$ 
and we recall an explicit integral expression of its heat kernel known as {\it Beals-Gaveau-Greiner formula}, cf. \cite{BGG1,BGG2,CCFI}. 
\vspace{1mm}\par 
Assuming the existence of a lattice $\Gamma$ in $\mathbb{G}$ we decompose the sub-Laplacian $\Delta_{\textup{sub}}$ on the compact nilmanifold 
$M=\Gamma \backslash \mathbb{G}$ into an infinite sum of elliptic operators acting on line bundles in Section \ref{section_3}. Via this method 
we obtain a decomposition of the heat trace of $\Delta_{\textup{sub}}^{\Gamma\backslash {\mathbb{G}}}$ into the heat traces of its component elliptic operators, cf. \cite{BFI2}, 
and we present a trace formula for the sub-Laplacian on $M$.
\vspace{1mm}\par 
In Section \ref{Section_Pseudo_H_type_algebras} we recall the notion of pseudo $H$-type Lie algebras and groups following \cite{Ci,FM1,FM}. We discuss the existence and some 
basic properties of integral lattices for such groups. These will play a role in our construction in Section \ref{Section_non_diffeomorphic_isospectral}. 
\vspace{1mm}\par 
In Section \ref{Section_pseudo_H_type_algebra}  we study the eigenvalues of a matrix-valued function which encodes the structure constants of the pseudo $H$-type 
Lie algebra. These data are essential in the calculation of the heat kernel of the sub-Laplacian in Section \ref{Section_Heat_kernel_Liegroups} and the trace formula in Section 
\ref{section_3}. Based on the trace formula we give a criterion for isospectrality of two pseudo $H$-type nilmanifolds in Section \ref{SpecH-typeNil} (Theorem \ref{multiple spectrum}). 
\vspace{1mm}\par
The last sections contain our main results. We use the classification of pseudo $H$-type Lie algebras in \cite{FM2,FM1} to construct finite families of isospectral, non-diffeomorphic 
pseudo $H$-type nilmanifolds. Finally, we present two nilmanifolds of different dimensions such that the short time heat trace expansions of the corresponding sub-Laplace operators 
coincide up to a term vanishing to infinite order as time tends to zero. 
\section{Heat kernel on two step nilpotent Lie groups}
\label{Section_Heat_kernel_Liegroups}
We recall the integral form of the heat kernel for a sub-Laplacian on simply connected two step nilpotent Lie groups given in \cite{BGG1,BGG2}, see also 
\cite{CCFI,Fu}. 
\subsection{Sub-Laplacian on two step nilpotent groups}
Let $\mathbb{G}$ be a simply connected two step nilpotent Lie group with Lie algebra $\mathcal{N}$. We assume that 
\begin{equation}\label{GL_assumtion_center}
\big{[}\mathcal{N},\mathcal{N}\big{]}=\mbox{\it center of } \mathcal{N}
\end{equation}
and we fix a basis $\{X_i,\, Z_k\:| \:  i=1, \cdots , N, \; k=1, \cdots, d\}$ of $\mathcal{N}$ such that $\{Z_k\}_{k=1}^{d}$ and $\{X_i\}_{i=1}^{N}$ span the center $[\mathcal{N},\mathcal{N}]$  and its 
complement, respectively. Moreover, we assume that $\mathcal{N}$ is equipped with an inner product with respect to which $\{X_i,\,Z_k\}$ 
becomes an orthonormal basis. Hence the Lie algebra $\mathcal{N}$ is decomposed into an orthogonal sum 
\[
\mathcal{N}=\textup{span}\big{\{} X_1, \cdots, X_N \big{\}}\oplus_{\perp}\big{[}\mathcal{N},\mathcal{N}\big{]}
\cong \mathbb{R}^{N}\oplus_{\perp}\mathbb{R}^{d},
\]
The expansion of Lie brackets
\begin{equation}\label{GL_definition_structure_constants}
[X_i,X_j]=\sum\limits_{k=1}^{d} c_{i\,j}^kZ_k,
\end{equation}
defines the structure constants $c_{i\,j}^k=-c_{j\,i}^k$. Given $z=\sum_{k=1}^d z_kZ_k\in [\mathcal{N}, \mathcal{N}]$ we denote by $\Omega(z)$ the skew-symmetric matrix 
\begin{equation}\label{Omega(z)}
\Omega(z)= \sum\limits_{k=1}^{d}\,z_k\,\big{(} c_{i\,j}^k \big{)}_{i,j}\in \mathbb{R}(N)=: \mbox{\it algebra of } N \times N \mbox{ \it real matrices}.
\end{equation}
\begin{rem}
{\it
Throughout the paper we identify the group $\mathbb{G}$ with $\mathbb{R}^{N}\times \mathbb{R}^{d}$ via the above coordinates, i.e. 
\[
\mathbb{G}\ni g= \sum_{i=1}^N x_i X_i + \sum_{k=1}^d z_k Z_k \longleftrightarrow (x_1,\cdots,x_N,z_1,\cdots,z_d)\in
\mathbb{R}^{N}\times\mathbb{R}^{d}\cong \mathcal{N}.
\]
Then the exponential map $\exp:\mathcal{N}\stackrel{\approx\,\,}\rightarrow\mathbb{G}$ is the identity. Via the Baker-Campbell Hausdorff 
formula and this identification we can express the group  product $*$ on $\mathbb{G} \cong \mathcal{N}$ in the form
\begin{equation*}
g * h =g+h+ \frac{1}{2}\big{[}g,h \big{]}. 
\end{equation*}
More explicitly and with respect to the above coordinates one has: 
\begin{multline*}
g*h=(x_1,\cdots,x_N,z_1,\cdots,z_d)*(x^{\prime}_1,\cdots,x_N^{\prime},z^{\prime}_1,\cdots,z_d^{\prime})\\
=\Big{(}x_1+x^{\prime}_1,\cdots,x_N+x^{\prime}_N,
z_1+z^{\prime}_1+\frac{1}{2}\sum_{i,j}
c_{i\,j}^{1}x_ix^{\prime}_j,\cdots,
z_d+z^{\prime}_d+\frac{1}{2}\sum_{i,j}\,c_{i\,j}^dx_ix^{\prime}_j\Big{)}.
\end{multline*}}
\end{rem}
\par 
Let $\tilde{X}_i$ denote the left invariant vector field on $\mathbb{G}$ corresponding to $X_i \in \mathcal{N}$ and consider the {\it sub-Laplacian}
\begin{equation}\label{sub-Laplacian}
\Delta_{\textup{sub}}^{\mathbb{G}}=-\frac{1}{2}\sum\limits_{i=1}^{N} {\tilde{X}_i}^{\,2}. 
\end{equation}
Based on (\ref{GL_assumtion_center}) the operator  $\Delta_{\textup{sub}}^{\mathbb{G}}$ is known to be sub-elliptic \cite{Ho} and essentially selfadjoint 
in $L_2(\mathbb{G})$ with respect to the Haar measure and considered on compactly supported smooth functions 
$C_{0}^{\infty}(\mathbb{G})$, cf. \cite{Stri86, Stri89}. 
\subsection{Beals-Gaveau-Greiner formula}
Next we recall the integral expression of the kernel function $K(t,g,h)\in C^{\infty}(\mathbb{R}_{+}\times\mathbb{G}
\times\mathbb{G})$ of the heat operator
\begin{equation}\label{GL_Definition_heat_operator}
e^{-t\,\Delta_{\textup{sub}}^{\mathbb{G}}}, 
\end{equation}
where $\mathbb{G}$ is a general 2-step nilpotent Lie group as above. The existence of a smooth kernel has been shown in \cite{Stri86, Stri89} and since the 
sub-Laplacian $\Delta_{\textup{sub}}^{\mathbb{G}}$ is a left-invariant operator it follows that $K$ is a convolution kernel, i.e. 
$$K^{\mathbb{G}}(t,g,h)=k^{\mathbb{G}}(t,\,g^{-1}\: * \: h)$$
with a smooth function $k^{\mathbb{G}}\in C^{\infty}(\mathbb{R}_{+}\times \mathbb{G})$.
\vspace{1mm}\par
In \cite{BGG1, BGG2, CCFI, Fu} an integral expression of $k^{\mathbb{G}}$ is given explicitly. Below we will calculate the 
spectrum of the sub-Laplacian on a class of nilmanifolds by using this expression. Recall that in the integrand of $k^{\mathbb{G}}$ two functions ({\it action and 
volume function}) appear. The integration is taken over a space which can be interpreted as the {\it characteristic variety} of the sub-Laplacian. Here we will neither 
present the details of this structure nor a proof of the next theorem. 
\begin{thm}[{Beals-Gaveau-Greiner formula}, \cite{BGG1,CCFI}]\label{heat kernel by BGG1} 
The integral kernel of the heat operator (\ref{GL_Definition_heat_operator}) has the form:
\[
K^{\mathbb{G}}(t,g,h)=k^{\mathbb{G}}(t,g^{-1}\: * \:  h)=\frac{1}{(2\pi
t)^{N/2+d}}\int_{\mathbb{R}^d}e^{-\frac{f(\tau,\,g^{-1}\: * \: h)}{t}}\,W(\tau)\,d\tau,
\]where the functions $f=f(\tau,g)\in C^{\infty}(\mathbb{R}^{d}\times
\mathbb{G})$ and $W(\tau)\in C^{\infty}(\mathbb{R}^d)$ are given as follows:
put $g=(x,z)\in \mathbb{R}^N\times\mathbb{R}^d$, then
\begin{align*}
&f(\tau,g)=f(\tau,x,z)=\sqrt{-1}\langle \tau,\,z \rangle 
+\frac{1}{2}\Big{\langle}\Omega(\sqrt{-1}\tau)\coth\bigr(\Omega(\sqrt{-1}\tau)\bigr)
\cdot x,x\Big{\rangle},\\
&W(\tau)=
\left\{\det\frac{\Omega(\sqrt{-1}\tau)}{\sinh\Omega(\sqrt{-1}\tau)}\right\}^{1/2},
\end{align*}
where $\langle z,z^{\prime}\rangle=\sum\limits_{k=1}^d z_k z_k^{\prime}$ denotes the Euclidean inner product on $\mathbb{R}^d$. 
\end{thm}
\begin{rem}
{\it 
Later on we will use the notation $ \langle \bullet,\,\bullet \rangle_{r,s}$ for a non-degenerate indefinite scalar product with the signature $(r,s)$ such that 
$\langle \bullet,\,\bullet \rangle=\langle \bullet,\,\bullet \rangle_{d,0}$.}
\end{rem}
We call $f=f(\tau,x,z)$ and $W(\tau)d\tau$ the {\it complex action function} and the {\it volume form}, respectively. Recall that $f$ is constructed by the {\it complex Hamilton-Jacobi method}, 
and the volume function $W(\tau)$ is sometimes referred to as {\it van Vleck determinant}. It is the Jacobian of the correspondence between the space of initial conditions and boundary conditions 
when we solve the Hamilton equation associated to the symbol of the sub-Laplacian. The solution can be interpreted as the bi-characteristic flow in the subriemannian setting. We recall that the volume 
function satisfies a {\it transportation equation}.
\section{Lattices and decomposition of a sub-Laplacian}
\label{section_3}
Based on Theorem \ref{heat kernel by BGG1}  we describe the heat  kernel of the sub-Laplacian descended to the quotient space $\Gamma\backslash \mathbb{G}$ 
(left coset space) by a lattice $\Gamma$. Such a space is called a (compact) {\it 2-step nilmanifold}. In the following we assume that there exists a {\it lattice} (cocompact discrete subgroup) 
in $\mathbb{G}$. We recall {\it Mal\'cev's Theorem}:
\begin{thm}[Mal\'cev, \cite{Ma,Ra}]
A nilpotent Lie group $G$ possesses a lattice $\Gamma$, i.e.  $\Gamma\backslash \mathbb{G}$ is compact, if and only if 
there exists a basis $\{Y_i\}$ in its Lie algebra $\mathfrak{g}$ such that the structure constants $\{\alpha_{i\,j}^k\}$ defined by
\[
[Y_i,\,Y_j]=\sum_k \alpha_{i\,j}^kY_k
\]
are all rational numbers.
\end{thm}
\subsection{Torus bundle and a family of elliptic operators}
\label{sub_section_3_1}
We recall a heat trace formula which previously has been obtained in \cite[Theorem 4.2]{BFI2}. Our analysis is essential based on this formula and in order to keep the paper 
self-contained we now repeat the main steps of the calculation. 
\vspace{1mm}\par 
Let $\Gamma$ be a lattice in a simply connected 2-step nilpotent Lie group $\mathbb{G}\cong \mathbb{R}^N\times\mathbb{R}^d$. 
The quotient space $\Gamma \backslash \mathbb{G}$ can be equipped with a subriemannian structure naturally inherited from that of $\mathbb{G}$. Its sub-Laplacian, which we now 
denote by $\Delta_{\textup{sub}}^{\Gamma\backslash {\mathbb{G}}}$, is the operator descended from the sub-Laplacian 
$\Delta_{\textup{sub}}^{\mathbb{G}}$ on $\mathbb{G}$.
\vspace{1mm}\par
For an element $g\in \mathbb{G}$ we will denote by $[g]\in \Gamma\backslash \mathbb{G}$  the corresponding class in the quotient space. 
Then, the heat kernel 
\[
K^{\Gamma\backslash {\mathbb{G}}}(t,[g],[h])
\in C^{\infty}\Big{(}\mathbb{R}_{+}\times \Gamma\backslash \mathbb{G}\times \Gamma\backslash \mathbb{G}\Big{)}
\]
of the sub-Laplacian $\Delta_{\textup{sub}}^{\Gamma\backslash {\mathbb{G}}}$ on the nilmanifold $\Gamma\backslash \mathbb{G}$ is given by
\begin{align}\label{heat_kernel_nilmanifold}
K^{\Gamma\backslash {\mathbb{G}}}(t,[g],[h])
&=\sum\limits_{\gamma\in \Gamma}K^{\mathbb{G}}\big{(}t, \gamma * g,h\big{)}\\
&=\sum\limits_{\gamma\in \Gamma}k^{\mathbb{G}}\big{(}t,g^{-1}*\gamma * h\big{)}
\in C^{\infty}\Big{(}\mathbb{R}_+ \times \Gamma \backslash \mathbb{G}\times  \Gamma \backslash \mathbb{G}\Big{)}.\notag
\end{align}
\par
Assuming the existence of a lattice $\Gamma$ in $\mathbb{G}$ we can decompose the sub-Laplacian into a family differential operators acting on invariant subspaces 
according to a torus bundle structure of  $\Gamma\backslash {\mathbb{G}}$. Next, we present some details and give the heat kernel expression for each component elliptic operator. 
\vspace{1mm}\par 
Let $\mathbb{A}\cong \mathbb{R}^d$ be the center of the group $\mathbb{G}$ where as before the identification is done with respect to the fixed orthonormal basis 
$\{ Z_k\}$ of $\mathbb{A}$. We obtain a principal bundle with the structure group
$\mathbb{A}/(\Gamma \cap \mathbb{A})\cong \mathbb{T}^{\dim \mathbb{A}}=\mathbb{T}^d$
\[
\Gamma\backslash {\mathbb{G}}
\longrightarrow (\Gamma/\Gamma\cap\mathbb{A})\backslash (\mathbb{G}/\mathbb{A})\cong
(\Gamma * \mathbb{A})\backslash \mathbb{G}.
\]
Note that the base space $(\Gamma/\Gamma\cap\mathbb{A})\backslash (\mathbb{G}/\mathbb{A})\cong
(\Gamma *\mathbb{A})\backslash \mathbb{G}$ is also a torus of dimension 
$\dim \mathbb{G} -\dim \mathbb{A}=N+d-d=N.$ Since $\mathbb{A}$ is abelian, the subgroup $\Gamma * \mathbb{A}$ coincides with 
$\Gamma+\mathbb{A}$, i.e. with the sum in the Lie algebra.
\vspace{1mm}\par 
Let ${\bf n}$ be an element in the ``dual lattice'' $[\Gamma\cap\mathbb{A}]^*$ of $\Gamma\cap\mathbb{A}$, that is,  
${\bf n}$ is a linear functional on $\mathbb{A}$ with the property that
\[
{\bf n}(\gamma)\in\mathbb{Z}~\mbox{\it for all}~\gamma\in\Gamma\cap\mathbb{A}.
\]
We may express ${\bf n}$ in the form ${\bf n}=\sum\limits_{k=1}^dn_kZ_k$ with integer coefficients $n_k\in\mathbb{Z}$ such that
\[
{\bf n}(\gamma)=\langle {\bf n}, \gamma\rangle=\sum n_k\langle Z_k,\gamma\rangle \in \mathbb{Z}~\mbox{\it for all}~\gamma\in\Gamma\cap\mathbb{A}.
\]
Then, the function space $C^{\infty}(\Gamma\backslash \mathbb{G})$ is decomposed {via a}  Fourier series expansion:
\[
C^{\infty}(\Gamma\backslash \mathbb{G})\ni {^{\forall} f}~;~f(g)
=\sum\limits_{{\bf n}\in [\Gamma\cap\mathbb{A}]^*}
\int_{\mathbb{T}^d}f(g * \lambda)
\overline{\chi_{\bf n}(\lambda)}{d\lambda,} 
\]
where $\chi_{\bf n}:\mathbb{T}^d\cong\mathbb{A}/(\Gamma \cap \mathbb{A})\to U(1)$ with $\chi_{\bf n}(\lambda)=e^{2\pi\sqrt{-1} \langle {\bf n},\lambda\rangle}$ is a 
unitary character corresponding to a dual element ${\bf n}\in [\Gamma\cap\mathbb{A}]^*$. So, we decompose 
\[
C^{\infty}(\Gamma\backslash \mathbb{G})
=\sum\limits_{{\bf n}\in [\Gamma\cap\mathbb{A}]^*}\mathcal{F}^{({\bf n})},
\]
where
\[
\mathcal{F}^{({\bf n})}
=\left\{\int_{\mathbb{T}^d}f(g * \lambda)\,\overline{\chi_{{\bf n}}(\lambda)} d\lambda~|~f\in C^{\infty}(\Gamma\backslash \mathbb{G})~\right\}.
\]
\par
The subspace $\mathcal{F}^{({\bf n})}$ can be seen as a space of smooth sections of a line bundle $E^{({\bf n})}$ on the base space 
$(\Gamma+\mathbb{A})\backslash \mathbb{G}\cong (\Gamma/\Gamma\cap\mathbb{A})\backslash(\mathbb{G}/\mathbb{A})$ associated to the
character $\chi_{{\bf n}}$. The sub-Laplacian leaves invariant each subspace $\mathcal{F}^{({\bf n})}$ and therefore it can be interpreted 
as a differential operator $\mathcal{D}^{({\bf n})}$ acting on the line bundle $E^{({\bf n})}$. Since the subbundle spanned by the (left)-invariant 
vector fields $\{\tilde{X}_i\: | \: i=1, \cdots, N\}$ defines a connection, i.e., its linear span is equivariant and transversal to the structure group action by 
$\mathbb{A}/(\Gamma \cap \mathbb{A})$, each operator $\mathcal{D}^{({\bf n})}$ is elliptic. Hence the sub-Laplacian 
$$\displaystyle{\Delta_{\textup{sub}}^{\Gamma\backslash\mathbb{G}}=-\frac{1}{2}\sum_i{\tilde{X}_i}^2}$$ 
can be seen as an infinite sum of elliptic operators. As a consequence we obtain a decomposition of the operator trace: 
\begin{equation}\label{GL_decomposition_of_operator_trace}
\text{{\bf tr}}\left(e^{-\Delta_{\textup{sub}}^{\Gamma\backslash\mathbb{G}}}\right)=
\sum\limits_{{\bf n}\in [\Gamma\cap\mathbb{A}]^*}\text{{\bf
tr}}\left(e^{-t\mathcal{D}^{({\bf n})}}\right).
\end{equation}
\par
Recall that $\{Z_k\:|\: k=1, \cdots, d\}$ denotes an orthonormal basis of the center $[\mathcal{N}, \mathcal{N}]$ of $\mathcal{N}$. As before we write 
$\tilde{Z}_k, k=1, \cdots, d$ for the corresponding left-invariant vector fields on the group $\mathbb{G}$. We equip $\mathbb{G}$ with a left-invariant 
Riemannian metric defined by assuming that the frame $[\tilde{X}_1, \cdots, \tilde{X}_N, \tilde{Z}_1, \cdots, \tilde{Z}_d]$ is orthonormal at any point of 
$\mathbb{G}$. Then the corresponding Laplacian has the form
\begin{equation}\label{Laplacian}
\Delta^{\mathbb{G}}=\Delta_{\textup{sub}}^{\mathbb{G}}-\frac{1}{2}\sum\limits_{k=1}^{d}\tilde{Z}_k^2. 
\end{equation}
The action of the difference $\Delta^{\mathbb{G}}-\Delta_{\textup{sub}}^{\mathbb{G}}$ on the subspace 
$\mathcal{F}^{({\bf n})}$ for each dual element ${\bf n}\in [\Gamma\cap\mathbb{A}]^{*}$ is given as follows:
\begin{prop}\label{action of central vector fields}
Let $f\in\mathcal{F}^{({\bf n})}$, then
\[
(\Delta^{\mathbb{G}}-\Delta_{\textup{sub}}^{\mathbb{G}})f=-\frac{1}{2}\sum\limits_{k=1}^{d}\tilde{Z}_k^2(f)=2\pi^2\sum\limits_{k=1}^{d}{n_k}^2\cdot\,f.
\]
\end{prop}
\subsection{Heat trace of the component  operators}
\allowdisplaybreaks{
Next we give an expression of the heat trace of each operator $\mathcal{D}^{({\bf n})}$. Recall that the heat kernel 
$K^{\Gamma\backslash\mathbb{G}}$ of $\Delta_{\textup{sub}}^{\Gamma\backslash {\mathbb{G}}}$ is given by 
(\ref{heat_kernel_nilmanifold}). Let $\mathcal{F}_{\Gamma}$ and $\mathcal{F}_{\Gamma\cap\mathbb{A}}$ be a fundamental domain for the lattice 
$\Gamma$ in $\mathbb{G}$ and $\Gamma\cap\mathbb{A}$ in the Euclidean space $\mathbb{A}$, respectively. Then the integral
\[
k_{\mathcal{D}^{({\bf n})}}\big{(}t,[g],[h]\big{)}
=\int_{\mathcal{F}_{\Gamma\cap\mathbb{A}}}K^{\Gamma\backslash\mathbb{G}}\big{(}t,[g],[h] *\lambda\big{)}
\overline{\chi_{\bf n}(\lambda)}\,d\lambda
\]
is the kernel function for the heat operator $e^{-t \mathcal{D}^{({\bf n})}}$, that is it satisfies
\begin{align*}
&k_{\mathcal{D}^{(\bf n)}}(t,[g] * \theta,[h])=k_{\mathcal{D}^{(\bf
n)}}(t,[g*\theta],[h])
=\overline{\chi_{\bf n}(\theta)}k_{\mathcal{D}^{(\bf n)}}(t,[g],[h]),\\
&k_{\mathcal{D}^{(\bf n)}}(t,[g],[h]*\theta)=k_{\mathcal{D}^{(\bf
n)}}(t,[g],[h*\theta])
=\chi_{\bf n}(\theta)k_{\mathcal{D}^{(\bf n)}}(t,[g],[h]),
\end{align*}
where $\theta\in \mathbb{A}$.
Let $\mathbb{M}=\{\mu_i\}$ be a set of complete representatives of the coset space
$\Gamma/(\Gamma\cap \mathbb{A})$, then
the trace of the heat operator $\displaystyle{e^{-t\mathcal{D}^{({\bf n})}}}$ is
given as follows:
\begin{prop}[see \cite{BFI2}] For each ${\bf n}$ in  the dual lattice $[\Gamma\cap\mathbb{A}]^*$ and with the heat kernel $K^{\mathbb{G}}$ of the 
sub-Laplacian on $\mathbb{G}$: 
\label{heat trace for component elliptic operators}
\begin{align*}
&{\textup{Vol}\bigr(\mathbb{A}/(\Gamma\cap\mathbb{A})\bigr)}
\cdot\text{{\bf tr}}\left(e^{-t \mathcal{D}^{(\bf n)}}\right)\,=\,
\int_{\mathcal{F}_{\Gamma}}\left(\sum\limits_{\gamma\in\Gamma}
\int_{\mathcal{F}_{\Gamma\cap\mathbb{A}}}
K^{\mathbb{G}}(t,g,\,\gamma *g*\lambda)\overline{\chi_{\bf n}(\lambda)}d\lambda\right)\, dg\\
&=\int_{\mathcal{F}_{\Gamma}}\left(
\sum\limits_{\mu \in\mathbb{M}}
\sum\limits_{\nu\in\Gamma\cap\mathbb{A}}
\int_{\mathcal{F}_{{\Gamma\cap\mathbb{A}}}} 
k^{\mathbb{G}}\big{(}t, g^{-1}* \mu* g *\nu*\lambda\big{)}
\overline{\chi_{\bf n}(\lambda)}d\lambda\right)\, dg\\
&=\int_{\mathcal{F}_{\Gamma}}\left(\sum\limits_{\mu \in\mathbb{M}}\int_{\mathbb{R}^d}k^{\mathbb{G}}\big{(}t,g^{-1}*\mu* g*\lambda\big{)}
\overline{\chi_{\bf n}(\lambda)}d\lambda\right)\, dg\\
&=\sum\limits_{\mu \in\mathbb{M}}
\int_{\mathcal{F}_{\Gamma}}\int_{\mathbb{R}^d}k^{\mathbb{G}}\big{(}t, g^{-1}*\mu * g*\lambda\big{)}
\overline{\chi_{\bf n}(\lambda)}d\lambda\, dg.
\end{align*}
\end{prop}
}
Applying Theorem \ref{heat kernel by BGG1} we can give a more concrete expression of the formula in Proposition 
\ref{heat trace for component elliptic operators}.
\allowdisplaybreaks{
For this purpose and for the sake of simplicity,
we assume that the structure constants $c_{i\,j}^k$ in (\ref{GL_definition_structure_constants}) are of the form
\[
\displaystyle{c_{i\,j}^k=\frac{2 q_{i\,j}^k}{p_{0}}}
\] 
with a common positive integer
$p_{0}\geq 1$ and integers $q_{i\,j}^k$. 
Then we fix a lattice $\Gamma$ 
\[
\Gamma:=\left\{\sum\limits_{1\leq i\leq N} m_iX_i
+\sum\limits_{1\leq k\leq d}\frac{\ell_{k}}{p_0}Z_k~\Bigr|~m_i,\,\ell_k\,\in
\mathbb{Z}\,\right\},
\]
and we choose the set $\mathbb{M}=\bigr\{\,\mu
=\sum\limits_{1\leq i\leq N}m_iX_i~|~m_i\in\mathbb{Z}\,\bigr\}$ 
of complete representatives of the quotient group
$(\Gamma\cap\mathbb{A})\backslash\Gamma=\Gamma/(\Gamma\cap\mathbb{A})$. 
For each fixed 
$${\bf n}=p_0\sum\limits\limits_{k=1}^{d} n_kZ_k\in [\Gamma\cap\mathbb{A}]^*,$$where $n_k\in\mathbb{Z}$ we have
\begin{align*}
&{\textup{Vol}\bigr(\mathbb{A}/(\Gamma\cap\mathbb{A})\bigr)}\cdot
\text{{\bf tr}}\left(e^{-t\mathcal{D}^{(\bf n)}}\right)\\
&=\frac{1}{(2\pi t)^{N/2+d}}\int_{\mathcal{F}_{\Gamma}}\sum\limits_{\mu\in\mathbb{M}}
\int_{\mathbb{A}}\int_{\mathbb{R}^d}e^{-\sqrt{-1}\frac{\langle[\mu,x]+\lambda,\tau\rangle}{t}}\cdot 
\varphi_{t}(\tau,\mu)d\tau\,
\overline{\chi_{\bf n}(\lambda)}\,d\lambda\,dx\,dz{\,=\,(*),}
\end{align*}
where the function $\varphi_t(\tau, \mu)$ in the integrand is given by: 
\begin{equation*}
\varphi_t(\tau, \mu)=\exp\Big{\{}  - \frac{1}{2t} \big{\langle} \Omega(\sqrt{-1} \tau) \coth \Omega(\sqrt{-1} \tau) \cdot \mu, \mu \big{\rangle} \Big{\}}
W(\tau). 
\end{equation*}
In the following we write $\hat {\varphi}_t(\tau, \mu)$ for the Fourier transform of $\varphi_t$ with respect to the $\tau$-variable. Then 
\begin{align*}
(*)
&=\frac{1}{t^{N/2+d}\cdot
(2\pi)^{(N+d)/2}}\int_{\mathcal{F}_{\Gamma}}\sum\limits_{\mu\in\mathbb{M}}
\int_{\mathbb{A}}
\hat{\varphi}_{t}\left(\frac{[\mu,x]+\lambda}{t},\mu\right)\cdot
e^{-2\pi\sqrt{-1}\langle {\bf n},\lambda\rangle}\,d\lambda\,dx\, dz\\
&=\frac{1}{t^{N/2+d} \cdot(2\pi)^{(N+d)/2}}
\int_{\mathcal{F}_{\Gamma}}\sum\limits_{\mu\in\mathbb{M}}
\int_{\mathbb{A}}
\hat{\varphi}_{t}\left(u, \mu\right)
\cdot e^{-2\pi\sqrt{-1}\langle {\bf n},\,tu+[x,\,\mu]\rangle}t^d\, du\,dx\, dz, \\
&=\frac{1}{(2\pi t)^{N/2}}\cdot {p_0}^d
\cdot\sum\limits_{\mu\in\mathbb{M}}\varphi_{t}(-2\pi t{\bf n},\mu)\cdot
\int_{{\underbrace{[0,1]\times\cdots\times [0,1]}_{N \textup{ times }}}}\,
e^{-2\pi\sqrt{-1}\langle {\bf n},\,[x,\,\mu]\rangle}\,dx.
\end{align*}
With a suitable set of linear independent vectors $a_{1}({\bf n}),$ $\cdots,a_{b({\bf n})}({\bf n})$ in $\Gamma$ the solution space $\mathbb{M}({\bf  n})=
\{\mu\in\mathbb{M}~|~\Omega({\bf n})(\mu)=0~\}$ can be written as
\[
\mathbb{M}({\bf  n})=\Bigr\{\,\mu=\sum\limits_{i=1}^{b({\bf n})}
m_i\,a_{i}({\bf n}),~\Bigr|~m_i\in\mathbb{Z}\,\Bigr\}. 
\]
Here $b({\bf n})\leq N$ and $b({\bf n})=N$ if and only if ${\bf n}=0$. Hence 
\begin{thm}\label{heat trace: final form}
For each ${\bf n}$ in  the dual lattice $[\Gamma\cap\mathbb{A}]^*$ and with the above notation: 
\begin{align}
{\text{\bf tr}}\,\Big{(}e^{-t\,{\mathcal{D}^{(\bf n)}}}\Big{)}
&=\frac{1}{(2\pi t)^{N/2}}
\sum\limits_{\mu\in\mathbb{M}({\bf n})}
e^{-\frac{\bigr<\Omega(2\pi t\sqrt{-1}{\bf n})\,\coth\Omega(2\pi\sqrt{-1}t{\bf
n}) \mu,\,\mu\bigr>}{2t}}
\sqrt{\det\frac{\Omega(2\pi \sqrt{-1}t{\bf n})}{\sinh \Omega(2\pi\sqrt{-1}t{\bf n})}}\notag\\
&=\frac{1}{(2\pi t)^{N/2}}
\sum\limits_{\mu\in\mathbb{M}({\bf n})} 
e^{-\frac{<\mu,\mu>}{2t}}
\sqrt{\det\frac{\Omega(2\pi \sqrt{-1}t{\bf n})}{\sinh \Omega(2\pi\sqrt{-1}t{\bf n})}}. \label{GL_theorem_trace_formula_D_n}
\end{align}
In particular, it holds: 
\begin{equation*}
\text{{\bf tr}}\left(e^{-\Delta_{\textup{sub}}^{\Gamma\backslash\mathbb{G}}}\right)=\frac{1}{(2\pi t)^{N/2}} 
\sum_{{\bf n} \in [\Gamma \cap \mathbb{A}]^*} \sum\limits_{\mu\in\mathbb{M}({\bf n})} 
e^{-\frac{<\mu,\mu>}{2t}}
\sqrt{\det\frac{\Omega(2\pi \sqrt{-1}t{\bf n})}{\sinh \Omega(2\pi\sqrt{-1}t{\bf n})}}.
\end{equation*}
\end{thm}
\begin{proof}
It suffices to show the second equation in (\ref{GL_theorem_trace_formula_D_n}). Note that the defining equation $\Omega({\bf n})(\mu)=0$ for $\mu \in \mathbb{M}({\bf n})$ implies: 
\[
\big{\langle} \Omega(2\pi t\sqrt{-1}{\bf n})\,\coth\Omega(2\pi\sqrt{-1}t{\bf n})(\mu),\,\mu\big{\rangle}=\big{\langle}\mu,\,\mu\big{\rangle}
=\sum \limits m_im_j\big{\langle}a_i({\bf  n}),a_j({\bf n})\big{\rangle}.
\]
The last statement follows from (\ref{GL_decomposition_of_operator_trace}) and  (\ref{GL_theorem_trace_formula_D_n}). 
\end{proof}
\begin{cor}
For ${\bf n}, -{\bf n}\in [\Gamma \cap \mathbb{A}]^*$ the traces ${\text{\bf tr}}\,\big{(}e^{-t\,{\mathcal{D}^{(\bf n)}}}\big{)}$
and ${\text{\bf tr}}\,\big{(}e^{-t\,{\mathcal{D}^{(-\bf n)}}}\big{)}$ coincide.
\end{cor}
\section{Pseudo $H$-type algebras and groups}
\label{Section_Pseudo_H_type_algebras}
For the rest of the paper we consider a specific subclass of all 2-step nilpotent Lie groups, the so called {\it pseudo $H$-type groups}. These are generalizations 
of Heisenberg type groups in \cite{Ka1,Ka2} and have been first introduced in \cite{Ci}. An extensive analysis of the structure and classification of pseudo $H$-type groups and 
their algebras can be found in the recent papers \cite{FM2, FM1,FM}. For completeness we recall the relevant definitions: 
\vspace{1ex}\par 
We write $\mathbb{R}^{r,s}$ for the Euclidean space $\mathbb{R}^{r+s}$ equipped with the non-degenerate scalar product 
\[
\langle x,\,y\rangle_{r,s}:=\sum\limits_{i=1}^rx_iy_i-\sum\limits_{j=1}^{s}x_{r+j}y_{r+j}.
\]
Consider the quadratic form $q_{r,s}(x)= \langle x, x \rangle_{r,s}$ and let $C\ell_{r,s}$ denote the Clifford algebra generated by $(\mathbb{R}^{r,s}, q_{r,s})$.  
We call a  $C\ell_{r,s}$-module $V$ {\it admissible}, if there is a non-degenerate bilinear form ($=$ scalar product) $\langle \bullet,\bullet\rangle_{V}$ on $V$ satisfying the following conditions:
\begin{itemize}
\item[(a)] There is a {\it Clifford module action} $J:C\ell_{r,s}\times V\rightarrow V: (z,X)\mapsto J_{z}X$, i.e. 
\begin{equation}\label{AdMo1}
 J_zJ_{z^{\prime}}+J_{z^{\prime}}J_z=-2 \langle z,z^{\prime}\rangle_{r,s} I \hspace{2ex} \mbox{\it for all} \hspace{2ex} z, z^{\prime} \in \mathbb{R}^{r,s}. 
\end{equation}
\item[(b)] For all $z \in \mathbb{R}^{r,s}$ the map $J_z$ is skew-symmetric on $V$ with respect to $\langle \bullet, \bullet \rangle_V$, i.e.
\begin{equation}\label{AdMo2}
 \big{\langle} J_{z}X,Y\big{\rangle}_V+\big{\langle} X,J_{z}Y \big{\rangle}_{V}=0 \hspace{2ex} \mbox{\it for all} \hspace{2ex} X,Y \in V. 
\end{equation}
\end{itemize}
Moreover,  from (a) and (b) one concludes: 
\begin{equation}\label{third_relation_Clifford_representation}
\big{\langle} J_zX, J_zY \big{\rangle}_V = \langle z,z \rangle_{r,s} \langle X,Y \rangle_V \hspace{2ex} \mbox{\it where} \hspace{2ex} X,Y \in V \hspace{1ex} z \in \mathbb{R}^{r,s}.  
\end{equation}
\par 
We write $\{J,~V,~\langle \bullet,\,\bullet \rangle_V\}$ for an admissible module of the Clifford algebra 
$C\ell_{r,s}$ with the module action $J=J_{z}$ and the scalar product $\langle \bullet,\,\bullet\rangle_V$.
\begin{rem}\label{Remark_admissible_modules_5_cases}
{The existence of an admissible $C\ell_{r,s}$-module $V$ has been shown in \cite{Ci}. If $s\not= 0$ then an admissible module $V$ needs not to be irreducible.  
More precisely, five cases are possible which all are present in the classification. If $C\ell_{r,s}$ has, up to equivalence, only one irreducible representation 
$(J,V)$, then either $V$ or the sum $V \oplus V$ is admissible. In the case where $C\ell_{r,s}$ has two non-equivalent irreducible representations $(J^{(i)}, V_i)$, $i=1,2$, then either 
$V_i$ for $i=1,2$ both are admissible, or only $V_1 \oplus V_2$ is admissible, or $V_1 \oplus V_1$ and $V_2 \oplus V_2$ simultaneously are admissible. These cases are 
complementary to each other (cf. \cite{Ci, FM2,FM1,FM}).
\vspace{1mm}\par 
In the case $s=0$ the situation is simpler.  Every irreducible module $V$ is admissible with respect to an inner product (i.e.  $\langle \bullet,\,\bullet\rangle_{V}$ is positive definite). 
Originally such cases have been defined and studied by A. Kaplan in \cite{Ka1}. }
\end{rem}
}
{In the following, we call a vector $X\in V$
{\it positive} (resp. {\it negative}) 
if the scalar product $\langle X,\,X\rangle_{V}$
is {\it positive} (resp. {\it negative}) and {\it null vector} if $\langle X, \: X \rangle_{V}=0$. A similar notation is used for vectors $Z \in \mathbb{R}^{r,s}$. 
{
If $s>0$, then an admissible module $V$ with scalar product $\langle
\bullet,\,\bullet\rangle_{V}$ 
has positive and negative subspaces of the same dimension $N$ with respect to the above scalar product $\langle \bullet,\,\bullet\rangle_{V}$, cf. \cite{Ci}. In particular, $\dim V=2N$ is even.
\vspace{1mm}\par 
Moreover, $V$ decomposes into the orthogonal sum of {\it minimal dimensional admissible modules}. In fact, since the scalar product restricted to 
such an invariant subspace is non-degenerate the orthogonal complement
is also an admissible module.}}

\begin{defn}\label{definition Pseudo}
Let $\{ J,V, \langle \bullet, \bullet \rangle_V \}$ be an admissible $C\ell_{r,s}$-module. 
\begin{itemize}
\item[(1)] The 
{2-step nilpotent} Lie algebra $V\oplus_{\perp}\mathbb{R}^{r,s}$ {with center $\mathbb{R}^{r,s}$ and} 
Lie brackets defined via the relation 
\begin{align}\label{bracket definition}
\big{\langle} J_z(X),\,Y \big{\rangle}_{V}=\big{\langle} z,[X,Y] \big{\rangle}_{r,s},  \hspace{2ex} z\in\mathbb{R}^{r,s}, ~\text{and}~X,Y\in V,
\end{align}
will be denoted by $\mathcal{N}_{r,s}(V)$. We write $\mathbb{G}_{r,s}(V)$  for 
{the corresponding simply connected Lie group and call it  a {\it pseudo $H$-type group}, cf.~\cite{Ci,FM}.}
\item[(2)]  If $V$ is of minimal dimension among all admissible modules, then we call $V$ {\it minimal admissible} and we shortly write $\mathcal{N}_{r,s}:=\mathcal{N}_{r,s}(V)$ and   
 $\mathbb{G}_{r,s}:= \mathbb{G}_{r,s}(V)$. 
\end{itemize}
\end{defn}
\begin{rem}\label{uniqueness}
Note that minimal admissible modules are cyclic and the nilpotent Lie algebra $\mathcal{N}_{r,s}$ is unique up
to isomorphisms, {even if the Clifford algebra $C\ell_{r,s}$ admits two non-equivalent irreducible modules (cf. \cite{FM}).}
\end{rem}

We fix an orthonormal basis $\{Z_k\}_{k=1}^{r+s}$ in $\mathbb{R}^{r,s}$, i.e. we assume  that: 
\begin{align*}
&\langle Z_i,\,Z_i\rangle_{r,s}=1~(i=1,\cdots,r), ~\langle Z_{r+j},\,Z_{r+j}\rangle_{r,s}=-1~(j=1,\cdots,s),~\text{and}\\
&\langle Z_i,\,Z_j\rangle_{r,s}=0~(i\not= j). 
\end{align*}
Let $\{J,~V, ~ \langle \bullet,\bullet \rangle_V\}$ be an admissible $C\ell_{r,s}$-module. 
\begin{thm}[cf. \cite{CrDo,FM}] \label{integral basis I}
Assume that $s>0$. Then there exists an orthonormal basis $\{X_i\}_{i=1}^{2N}$ of $V$
such that
\begin{align*}
&\text{{\em (1)}}\quad \langle X_i,X_i\rangle_{V}=1~(i=1\cdots,N),~\langle X_i,X_i\rangle_{V}=-1~(i=N+1,\cdots,2N)~\text{and}\\
&\quad\quad  \langle X_i,X_j\rangle_V=0~\text{for}~i\not=j,\\
&\text{{\em (2)}}\quad \text{For each $k$, the operator $J_{Z_k}$ maps $X_i$ to some $X_j$ or $-X_j$ with $j\not=i$}.
\end{align*}
\end{thm}
\begin{defn}\label{integral basis}
We call a basis $\{X_i,Z_j\}$ satisfying the properties 
in Theorem \ref{integral basis I} an {\it integral basis} of the algebra $\mathcal{N}_{r,s}(V)$. 
\end{defn}
{\begin{rem}
An interesting problem, which we will postpone to a future work, consists in a classification of {\it integral bases} up to isomorphisms. Consider an orthonormal basis $\{Z_k\}_{k=1}^{r+s}$ of $\mathbb{R}^{r,s}$ in the above sense. 
If $V$ is a minimal admissible $\textup{C$\ell$}_{r,s}$-module, then we can define a finite subgroup $\mathbb{G}$ in $\textup{GL}(V)$ generated by 
$\{J_{Z_k} \: : \: k =1, \ldots, r+s\}$. Consider the commutative subgroup: 
\begin{equation*}
\mathbb{S}: =\big{\{} A \in \mathbb{G} \: : \: A^2= \textup{Id},\; A= J_{Z_{i_1}}\ldots J_{Z_{i_r}} >0, A \ne - \textup{Id} \big{\}} \subset \mathbb{G}. 
\end{equation*}
By "$A>0$" we mean that $A$ maps positive (resp. negative) vectors in $V$ to positive (resp. negative) vectors. 
Such groups are partially ordered with respect to the inclusion and we assume that $\mathbb{S}$ is a maximal element. Further, we assume that $v \in V$ is a common eigenvector 
of elements in $\mathbb{S}$. Necessarily, $v$ is not a null vector, 
i.e. $\langle v,v \rangle_{V} \ne 0$. Consider 
\begin{equation*}
\big{\{} \pm X_i \} = \big{\{} Av\: : \: A \in \mathbb{G} \big{\}}. 
\end{equation*}
We conjecture that a suitable choice of the common eigenvector $v$ leads to an integral basis $\{ X_i, Z_{\ell} \}$ of the pseudo $H$-type Lie algebra $\mathcal{N}_{r,s}(V)$. 
\vspace{1mm}\par
Conversely, let $\{ X_i, Z_j\}$ be an integral basis and put $\pm \mathcal{B}:=\{\pm  X_{\ell} \: : \: \ell = 1, \ldots , m= \dim V \}$. Then each $J_{Z_i}$ defines a bijective map 
\begin{equation*}
J_{Z_i}: \pm \mathcal{B} \rightarrow \pm \mathcal{B} 
\end{equation*}
and elements in the group $\mathbb{G}$ act on $\pm \mathcal{B}$. We obtain a subgroup $\mathbb{S}$ as above from this basis by  defining
\begin{equation*}
\big{\{} A \in \mathbb{G} \: : \: A(X_1)=X_1 \big{\}} =: \mathbb{S} \subset \mathbb{G}. 
\end{equation*}
\par 
We conjecture that every maximal subgroup $\mathbb{S}$ defines an integral basis. A classification of integral bases up to isomorphisms 
is left as an interesting problem, which we postpone to a future study. 
%
%
\end{rem}
}
From now on we assume that $\{X_i, Z_k\}$ is an integral basis of $\mathcal{N}_{r,s}(V)$. 
\begin{cor}\label{Corollary_mapping_properties_representation_integral_basis}
If there exists $i\in \{ 1, \cdots, 2N\}$ such that 
$J_{Z_k}(X_i)=\pm J_{Z_{\ell}}(X_i)$, then $k=\ell$. Hence any basis vector $X_i$ is mapped to some $X_j$ or $-X_j$ by at most one operator $J_{Z_k}$. 
\end{cor}
\begin{proof}
If $k \leq r$ then $J_{Z_k}$ maps positive to positive and negative to negative elements. Similarly, if $k>r$, then $J_{Z_k}$ maps
positive to negative and negative to positive elements. Therefore, under the above assumption only the cases $k,\ell\leq r$ or $k,\ell>r$ are possible.
\vspace{1mm}\par 
Let us  assume $k\not=\ell$ such that  $\pm X_i=J_{Z_k}J_{Z_{\ell}}(X_i)$. By the previous remark we have 
$$J_{Z_k}J_{Z_{\ell}}\circ J_{Z_k}J_{Z_{\ell}}=- J_{Z_k}\circ J_{Z_{\ell}}^2 \circ J_{Z_k}=-\big{\langle} Z_k,Z_k\big{\rangle}_{r,s} \big{\langle} Z_{\ell},Z_{\ell} \big{\rangle}_{r,s}=-I.$$
This equation contradicts the existence of the eigenvalue $1$ or $-1$ of $J_{Z_k}J_{Z_{\ell}}$.
\end{proof}
\begin{cor}\label{integral basis II}
If we put $[X_i,\,X_j]=\sum c_{i\,j}^kZ_k$, then $c_{i\,j}^{k}$ can be non-zero for at most one $k$. If $c_{i\,j}^{k}$ is non-zero then it equals $\pm 1$.  
\end{cor}
\begin{proof}
The statement follows from Corollary \ref {Corollary_mapping_properties_representation_integral_basis} and
\[
\big{\langle} J_{Z_\ell} X_i,\,X_j\big{\rangle}_{V}=\big{\langle} Z_\ell,\,[X_i,\,X_j]\big{\rangle}_{r,s}=
\begin{cases}
c_{i\,j}^{\ell} & \text{if } \ell \leq r\\
- c_{i\,j}^{\ell} & \text{if } \ell > r. 
\end{cases}
\]
\end{proof}
\begin{defn}\label{integral lattice}
From an integral basis $\{X_i,Z_k\}$ of $\mathcal{N}_{r,s}(V)$ we define a {\it lattice} in the pseudo $H$-type group $\mathbb{G}_{r,s}(V)$ by
\[
\Gamma_{r,s}(V):=\left\{\sum\limits_{m_i\in\mathbb{Z}}\,m_iX_i
+\frac{1}{2}\sum\limits_{k_j\in\mathbb{Z}} \,k_jZ_j\right\}. 
\]
In the following we call $\Gamma_{r,s}(V)$ a {\it standard integral lattice} in $\mathcal{N}_{r,s}(V)$. If  $\mathcal{N}_{r,s}$ is constructed from 
a minimal admissible module $V$ (cf. Definition \ref{definition Pseudo}), then we write $\Gamma_{r,s}:=\Gamma_{r,s}(V)$. 
\end{defn}
\begin{rem}
A  {\it standard integral lattice} is not unique. A complete classification will be subject of another work. 
For particular cases the construction of $\Gamma_{r,s}$ is found in \cite{FM}. 
\end{rem}
\par
In the following two sections we consider the sub-Laplacian 
\begin{equation}\label{GL_sub_Laplacian_last_two_sections}
\Delta_{\textup{sub}}^{\mathbb{G}_{r,s}(V)}=-\frac{1}{2}\sum_{i=1}^{2N} {\tilde{X}_{i}}^2
\end{equation}
on $\mathbb{G}_{r,s}(V)$, where $\{X_i\: : \: i=1, \cdots, 2N\}$ is the basis of the module $V$ in the definition of the standard integral lattice $\Gamma_{r,s}(V)$. 
We determine the heat trace of the sub-Laplacian 
\begin{equation}\label{sub_Laplacian_nilmanifold_definitionand notation}
\Delta_{\textup{sub}}^{\Gamma_{r,s}(V)\backslash\mathbb{G}_{r,s}(V)}
\end{equation}
descended from (\ref{GL_sub_Laplacian_last_two_sections}) to the nilmanifold $\Gamma_{r,s}(V)\backslash \mathbb{G}_{r,s}(V)$. Based on the sub-ellipticity of  
(\ref{sub_Laplacian_nilmanifold_definitionand notation}) it is known that the spectrum of  the sub-Laplacian only consists of eigenvalues with finite multiplicities. 
In principle our trace formula in Theorem \ref{heat trace: final form} can be used to obtain the spectrum of (\ref{sub_Laplacian_nilmanifold_definitionand notation}). 
However, we will not calculate the eigenvalues and multiplicities explicitly since a comparison of heat traces is sufficient to decide isospectrality. 
\section{The structure constants of pseudo $H$-type groups}
\label{Section_pseudo_H_type_algebra}
In the case of pseudo $H$-type groups we calculate the characteristic polynomial of the matrix $\Omega(z)$ in (\ref{Omega(z)}) in the case where $s>0$ in Definition 
\ref{definition Pseudo} of the pseudo $H$-type group $\mathbb{G}_{r,s}(V)$. Recall that this matrix is an essential ingredient for the integral expression of the heat kernel in 
Theorem \ref{heat kernel by BGG1}.
\vspace{1mm}\par 
Throughout this section we assume that $s>0$ so that we can use the integral basis in Theorem \ref{integral basis I}. Let us start by decomposing the Clifford module 
$V=V_+ \oplus_{\perp} V_-$ into a positive and a negative subspace $V_+:=[X_i \: : \: i=1, \cdots, N]$ and $V_-=[X_{N+j}\: : \: j=1, \cdots, N]$, where $\{X_i \: : \: i=1, \cdots, 2N \}$ 
is part of a standard integral basis $\{ X_i, Z_j\}$  of $\mathcal{N}_{r,s}(V)$. Then the Clifford module action 
\begin{equation*}
J_{z}(X_i)=\sum_j c_{i\,j}(z)X_{j}
\end{equation*}
with 
\[
z=\sum\limits_{i=1}^r \,\mu_iZ_i+\sum\limits_{j=1}^s\nu_{j}Z_{r+j}\cong (\mu,\nu)^T \in \mathbb{R}^{r,s}
\]
can be written in form of a matrix with respect to the basis $\{X_i\}$ of $V$:
\[
J_{z}=\begin{pmatrix}A&B\\C&D\end{pmatrix}: 
\begin{array}{l}
V_+ \\
\oplus_{\perp}\\
V_-
\end{array}
\longrightarrow 
\begin{array}{l}
V_+ \\
\oplus_{\perp}\\
V_-
\end{array}.
\]
\par
By (\ref{third_relation_Clifford_representation}) the map $J_z$ leaves $V_{\pm}$ invariant whenever $z$ is positive in $\mathbb{R}^{r,s}$. 
If $z\in \mathbb{R}^{r,s}$ is negative then $J_z$ maps $V_+$ to $V_-$ and vice versa. It follows that the component matrices $A, B, C,D$ are 
of the forms $A=A(\mu)$, $B=B(\nu)$, $C=C(\nu)$ and $D=D(\mu)$. Due to the admissibility condition (\ref{AdMo2}) of the Clifford action on the module $V$ we have 
$${A^T(\mu)}=-A(\mu),\; \;  {B^T(\nu)}=C(\nu)\; \textup{ \it  and }\;  {D^T(\mu)}=-D(\mu).$$ 
\par 
Here ${A^T(\mu)}$ denotes the transposed matrix of $A(\mu)\in \mathbb{R}(N)$.  
Moreover, the identity $J_{z}^2=-\langle z,z\rangle_{r,s}$ yields additional relations between the component matrices $A,B, C$ and $D$ which are collected in the next lemma. 
\begin{lem}\label{relations}
With the notion $\| \mu\|^2:= \sum_{i=1}^r\mu_i^2$ and $\| \nu\|^2:= \sum_{j=1}^s \nu_j^2$ we have: 
\begin{itemize}
\item[(a)] $A(\mu)^2+B(\nu)C(\nu)=-\langle z,z\rangle_{r,s}=-(\Vert\mu\Vert^2-\Vert\nu\Vert^2)$,
\item[(b)] $A(\mu)B(\nu)+B(\nu)D(\mu)=0$ and $C(\nu)A(\mu)+D(\mu)C(\nu)=0$, 
\item[({c})] $C(\nu)B(\nu)+D(\mu)^2=-\langle z,z\rangle_{r,s}$.
\end{itemize}
In particular, it follows $A(\mu)^2=-\Vert\mu\Vert^2$ and $D(\mu)^2=-\Vert\mu\Vert^2$.
\end{lem}
Let $L$ be a linear map on $V\cong \mathbb{R}^{2N}$. The same notation is used for its matrix representation with respect to the basis $\{X_i \}$. Let $\langle \cdot , \cdot \rangle$ denote the Euclidean 
inner product on $\mathbb{R}^{2N}$ and fix $x, y \in V$. We calculate the matrix representation of the transpose $L^*$ of $L$ with respect to the scalar product $\langle \cdot , \cdot \rangle_V$. Consider 
the matrix
\begin{equation*}
\tau:=  
\left(\begin{array}{cc}
I & 0 \\
0 & -I 
\end{array}\right) \in \mathbb{R}(2N) \hspace{2ex} \mbox{\it where} \hspace{2ex} I = \textup{identity} \in \mathbb{R}(N). 
\end{equation*}
Then we have 
\begin{equation*}
\big{\langle} Lx,y \big{\rangle}_V=\big{\langle} Lx, \tau y \big{\rangle}= \big{\langle} x,L^T\tau y \big{\rangle}=\big{\langle} x, \tau L^T \tau y \big{\rangle}_V, 
\end{equation*}
which implies that $L^*= \tau L^T \tau$. A direct calculation shows that the skew-symmetric matrix $\Omega(z)$ in (\ref{Omega(z)}) is related to the above matrix representation of $J_z$ as follows: 
\begin{equation}\label{relation_Omega_z_J_z}
\Omega(z)= \tau J_z^T=\left(
\begin{array}{cc}
-A(\mu) & B(\nu) \\
-B^T(\nu) & D(\mu)
\end{array}
\right). 
\end{equation}
\begin{rem} If $r=0$, then $A(\mu)=D(\mu)=0$.
\end{rem}
In order to determine the eigenvalues of the matrix $\Omega(\sqrt{-1}z)$ we employ the relation
\begin{align*}
\underbrace{\begin{pmatrix}-A+\lambda&B \\-C&D+\lambda\end{pmatrix}}_{=\Omega({z})+\lambda }
\begin{pmatrix} I&-(-A+\lambda)^{-1}B\\0& I\end{pmatrix}
&=\begin{pmatrix}-A+\lambda&0\\-{B^T}&{B^T}(-A+\lambda)^{-1}B+\lambda+ D\end{pmatrix}.
\end{align*}
Hence we have
\[
\det \bigr(\Omega(z)+\lambda\bigr) =\det \bigr(-A+\lambda\bigr)\: 
\det\Big{(}{B^T}(-A+\lambda)^{-1}B+\lambda+D\Big{)}.
\]
According to Lemma \ref{relations} one has $B^TB=\Vert\nu\Vert^2=B B^T$ and $B^TAB+\Vert\nu\Vert^2 D=0$, 
showing that $B^T\bigr(-A+\lambda\bigr)B=\Vert\nu\Vert^2 (\lambda + D)$. Together with the  skew-symmetry of $\Omega(z)$: 
\begin{align*}
&\det\bigr(\Omega(z)+\lambda\bigr)=\det\bigr(\Omega(z)-\lambda\bigr)
\\
&=\det \bigr(-A+\lambda\bigr)\: 
\det\Bigr(B^T(-A+\lambda)^{-1}B+\frac{1}{\Vert\nu\Vert^2}\: {B}^T\left(-A+\lambda\right){B}\Bigr)\\
&=\det \bigr(-A+\lambda\bigr)\: 
\det\left(B^T\Bigr[(-A+\lambda)^{-1}+\frac{1}{\Vert\nu\Vert^2}(-A+\lambda)\Bigr]B\right)\\
&=\det\Bigr(\Vert\nu\Vert^2-\Vert\mu\Vert^2+\lambda^2-2\lambda A\Bigr).
\end{align*}
Therefore:
\begin{align*}
&\det\bigr(\Omega(z)+\lambda\bigr)^2=\det\bigr(\Omega(z)+\lambda\bigr)\: \det\bigr(\Omega(z)-\lambda\bigr)\\
&=\det\Bigr(\left(\Vert\nu\Vert^2-\Vert\mu\Vert^2+\lambda^2\right)^2+4\lambda^2\Vert\mu\Vert^2\Bigr).
\end{align*}
\begin{prop}\label{characteristic polynomial} 
With $s>0$ and $z \in \mathbb{R}^{r,s}$ we have: 
\begin{align*}
\det (\Omega(z)+\lambda)^2
&=\Bigr(\bigr(\lambda^2+\Vert\mu\Vert^2+\Vert\nu\Vert^2\bigr)^2  -4\Vert\mu\Vert^2 \Vert\nu\Vert^2 \Bigr)^{N}\\
&=\Big{[} \big{(}\lambda^2+(\| \mu\|+ \|\nu\|)^2 \big{)} \big{(} \lambda^2+ (\| \mu\| - \|\nu\|)^2 \big{)} \Big{]}^N. 
\end{align*}
\end{prop}
By replacing $z$ with $\sqrt{-1}z$, $\mu$ with $\sqrt{-1}\mu$ and $\nu$ with $\sqrt{-1}\nu$ we have: 
\begin{cor}\label{eigenvalues}
The eigenvalues $\lambda$ of the matrix $\Omega(\sqrt{-1}z)$ are $\lambda=\pm(\Vert\mu\Vert\pm\Vert\nu\Vert)$.
\end{cor}
If $z\not= 0$ then the matrix $\Omega(z)$ has the eigenvalue zero only when $\Vert\mu\Vert=\Vert\nu\Vert$. 
In this case the matrices $B(\nu)$ and $D(\mu)$ are non-singular so that the dimension of the solution space $\Omega(z)\cdot x=0$ is $N$ 
($=$ half the dimension of $V$). 
\begin{prop}\label{zero solution space}
Assume that $z\not= 0$ and $\Vert\mu\Vert=\Vert\nu\Vert$. The kernel of $\Omega(z)$ is given by
\[
\Ker\Omega(z)=
\left\{
~\begin{pmatrix}B(\nu)x\\-D(\mu)x
\end{pmatrix}\:
~\Biggr|~ {x}=(x_1,\ldots,x_{N})^T\in\mathbb{R}^{N}
~\right\}.
\]
\end{prop}
Finally we determine the dimension of the eigenspaces corresponding to the above eigenvalues $\pm(\Vert\mu\Vert\pm \Vert\nu\Vert)$.  
\begin{prop}\label{eigenspace}
Let  $z\not= 0$. Then the dimensions of the eigenspaces $E_{\lambda}$ of $\Omega(\sqrt{-1}z)$ with respect to the eigenvalue 
$\lambda=\pm(\Vert\mu\Vert\pm\Vert\nu\Vert)$ are given as follows: 
\begin{itemize}
\item[(i)] If neither $\mu$ nor $\nu$ is zero, then $\dim E_{\lambda}=N/2$ and therefore $\frac{\dim V}{2}=N$ is even.
\item[(ii)] If $\mu=0$ and $\nu \ne 0$, then $\dim E_{\|\nu\|}= N= \dim E_{-\|\nu\|}$. 
\item[(iii)] If $\mu \ne 0$ and $\nu=0$, then  $\dim E_{\|\mu\|}= N=
  \dim E_{-\|\mu\|}$. 
\end{itemize}
\end{prop}
\begin{proof}
(i): Let $(x,y)^T \in \mathbb{R}^N \times \mathbb{R}^N \cong \mathbb{R}^{2N}$ be an eigenvector of the matrix $\Omega(\sqrt{-1}z)$ 
with respect to the eigenvalue $\lambda=\|\mu\|+\|\nu\|$, where $\mu,\nu\not=0$, then
\begin{equation}\label{GL_eigenvector_equation}
\Omega(\sqrt{-1}z)\:\begin{pmatrix}{x}\\{y}\end{pmatrix}=
\begin{pmatrix}-A(\sqrt{-1}\mu)&B(\sqrt{-1}\nu)\\
-C(\sqrt{-1}\nu)&D(\sqrt{-1}\mu)\end{pmatrix}\: \begin{pmatrix}{x}\\{y}\end{pmatrix}=
\bigr(\Vert\mu\Vert+\Vert\nu\Vert\bigr)\begin{pmatrix}{x}\\{y}\end{pmatrix}.
\end{equation}
Multiplying the first equation by $A(\sqrt{-1}\mu)$ and the second equation by $B(\sqrt{-1}\nu)$ gives
\begin{align*}
A(\sqrt{-1}\mu)\Big{(}-A(\sqrt{-1}\mu)\:{x}+B(\sqrt{-1}\nu)\:{y}\Big{)}
&=(\Vert\mu\Vert+\Vert\nu\Vert)A(\sqrt{-1}\mu)\:{x}  \\
B(\sqrt{-1}\nu)\Big{(}-C(\sqrt{-1}\nu)\:{x}+D(\sqrt{-1}\mu)\:{y}\Big{)}
&=(\Vert\mu\Vert+\Vert\nu\Vert)B(\sqrt{-1}\nu)\:{y}. 
\end{align*}
On the left hand side we use Lemma \ref{relations} and deduce the following two equations
\begin{align*}
-\Vert\mu\Vert^2\:{x}+A(\sqrt{-1}\mu)B(\sqrt{-1}\nu)\:{y}
&=(\Vert\mu\Vert+\Vert\nu\Vert)A(\sqrt{-1}\mu)\:{x}\\
\Vert\nu\Vert^2\:{x}+B(\sqrt{-1}\nu)D(\sqrt{-1}\mu)\:{y}
&=(\Vert\mu\Vert+\Vert\nu\Vert)B(\sqrt{-1}\nu)\:{y}.
\end{align*}
Adding these identities and using Lemma \ref{relations}, (ii) gives
\begin{equation*}
\big{(} \| \nu\| -\| \mu \|\big)\: x=A(\sqrt{-1} \mu) \: x+B(\sqrt{-1} \nu) \: y. 
\end{equation*}
Together with the Equation (\ref{GL_eigenvector_equation}) we find that $B(\sqrt{-1}\nu) y =\Vert\nu\Vert x$.  This shows that 
\[
A(\sqrt{-1}\mu)x  =-\Vert\mu\Vert x. 
\]
\par 
Since $B(\nu)$ is non-singular for $\nu\not=0$ the vector $y$ is uniquely determined by $x$. Conversely, the eigenvector $x$ of the matrix $A(\sqrt{-1}\mu)$ with the eigenvalue $-\Vert\mu\Vert\ne 0$, 
determines the eigenvector of the matrix $\Omega(\sqrt{-1}z)$ by putting ${y}=\Vert\nu\Vert {B(\sqrt{-1}\nu)}^{-1} x$.
\vspace{1mm}\par 
Since $A(\mu)$ is skew-symmetric for any real vector $\mu$ and ${A(\mu)}^2= -{\Vert\mu\Vert}^2$, the dimension of the eigenspace of the matrix
$A(\sqrt{-1}\mu)$ with respect to the eigenvalue $\Vert\mu\Vert$ is half the size of the matrix $A$, i.e. it equals $\frac{N}{2}$. The remaining eigenvalues can be 
treated similarly and therefore (i) follows. 
\vspace{1ex}\\
(ii): Under the assumption of (ii) and by applying the relations in Lemma \ref{relations} we have $A=D=0$ and $B({\nu})C(\nu)=\|\nu\|^2=C(\nu)B(\nu)$.  Let 
$(x,y)^T\in \mathbb{R}^N \times \mathbb{R}^N \cong \mathbb{R}^{2N}$ be an eigenvector of $\Omega(\sqrt{-1}z)$ with eigenvalue $\lambda=\|\nu\|$. Then the equation
\begin{equation*}
\Omega(\sqrt{-1}z)
\left(
\begin{array}{c}
x\\
y
\end{array}\right)= 
\left( 
\begin{array}{cc}
0 & B(\sqrt{-1}\nu)\\
\|\nu\|^2B^{-1}(\sqrt{-1}\nu) & 0
\end{array}
\right)
\left(
\begin{array}{c}
x\\
y
\end{array}
\right) = 
\|\nu\|
\left(
\begin{array}{c}
x\\
y
\end{array}
\right) 
\end{equation*}
is equivalent to $B(\sqrt{-1}\nu){y}=\|\nu\|{x}$, which can be uniquely solved for any given $x \in \mathbb{R}^N$. The case $\lambda=-\|\nu\|$ is treated in the 
same way and (ii) follows. 
\vspace{1ex}\\
(iii): If $\mu\ne 0$ and $\nu=0$, then $C=B=0$ and with $\lambda=\|\mu\|$ we have the equation
\begin{equation*}
\Omega(\sqrt{-1} z)
\left(
\begin{array}{c}
x\\
y
\end{array}\right)= 
\left( 
\begin{array}{cc}
-A(\sqrt{-1}\mu) & 0\\
0 & D(\sqrt{-1}\mu)
\end{array}
\right)
\left(
\begin{array}{c}
x\\
y
\end{array}
\right) = 
\|\mu\|
\left(
\begin{array}{c}
x\\
y
\end{array}
\right). 
\end{equation*} 
\par 
The matrix $\Omega(\sqrt{-1}z)$ is Hermitian and therefore can be diagonalized. From the expression of $\text{det }(\Omega(z)+\lambda)^2$, we deduce that the eigenspaces have dimension $N$.
The case $\lambda=-\|\mu\|$ can be treated similarly.  
\end{proof}
\begin{cor}\label{Corollary_det_Omega_plus_lambda_s>0}
Let $s>0$, then the characteristic polynomial of $\Omega(\sqrt{-1}z)$ is given by:
\[
\det{\Big{(}\Omega(\sqrt{-1}z)+\lambda\Big{)}}= \bigr(\lambda^2-(\Vert\mu\Vert+\Vert\nu\Vert)^2\bigr)^{N/2}
\bigr(\lambda^2-(\Vert\mu\Vert-\Vert\nu\Vert)^2\bigr)^{N/2}.
\]
\end{cor}
\begin{rem}
In the case $s=0$ we can find an admissible module $V$ with respect to an inner product ($=$ positive definite scalar product) and so we obtain: 
$$\Omega(z)=J_z\hspace{3ex} \mbox{\it  and} \hspace{3ex} \Omega(z)^2=J_z^2=-\|z\|^2 I.$$
Therefore, the statement of Corollary  \ref{Corollary_det_Omega_plus_lambda_s>0} remains valid even in this case: 
$$\text{det }\left(\Omega(\sqrt{-1}z)+\lambda\right)=\left(\lambda^2-\|z\|^2\right)^{N}.$$
\end{rem}
\section{Spectrum of the sub-Laplacian on pseudo $H$-type nilmanifolds}
\label{SpecH-typeNil}
Let $\{J,~V,~ \langle \bullet,\,\bullet \rangle_V\}$ be an admissible module of the Clifford algebra $C\ell_{r,s}$. Based on the results of the previous sections  
we derive an explicit expression for the heat trace of the sub-Laplacian $\Delta_{\textup{sub}}^{\Gamma_{r,s}(V)\backslash\mathbb{G}_{r,s}(V)}$ 
on the nilmanifolds $\Gamma_{r,s}(V)\backslash \mathbb{G}_{r,s}(V)$ for $r>0$ and $s>0$. 
In fact, with a view to the decomposition (\ref{GL_decomposition_of_operator_trace}), it suffices to calculate the heat trace of each component operator $\mathcal{D}^{({\bf n})}$ 
with respect to the element ${\bf n}$ in the dual $[\mathbb{A}\cap\Gamma_{r,s}(V)]^*$ of the lattice $\mathbb{A}\cap\Gamma_{r,s}(V)$. 
\vspace{1mm}\par 
Recall that $\mathbb{A}\cong\mathbb{R}^{r,s}$ denotes the center of the group $\mathbb{G}_{r,s}(V)$. However, 
our notation will not indicate the dependence on the parameter $(r,s)$.
\subsection{Determination of the spectrum}
An element ${\bf n}$ in the dual lattice $\bigr[\Gamma_{r,s}\cap \mathbb{A}\bigr]^*$ can be expressed as
\[
{\bf n}=2\left(\sum\limits_{i=1}^{r}
m_iZ_i+\sum\limits_{j=1}^sn_jZ_{r+j}\right) \hspace{2ex} \mbox{\it where} \hspace{2ex} (m,n) \in \mathbb{Z}^{r+s}. 
\]
We also use the notation ${\bf n}=2(\mu+\nu)$ with $\mu+\nu=(m_1,\cdots,m_r,n_1\cdots,n_s)\in\mathbb{Z}^{r+s}$. Now Theorem \ref{heat trace: final form} implies: 
\vspace{1ex}\\
(1)\quad If ${\bf n}=0$, then the trace of the operator {$e^{-t\mathcal{D}^{(0)}}$} is given by
\begin{equation}\label{trace1}
\text{{\bf tr}}\Big{(} e^{-t\mathcal{D}^{(0)}}\Big{)}
=\frac{1}{(2\pi\,t)^{N}}\sum\limits_{{\bf \ell}\in\mathbb{Z}^{2N}} e^{-\frac{\Vert{\bf \ell}\Vert^2}{2t}}.
\end{equation}
\noindent
(2)\quad Assume that ${\bf n}\in \bigr[\Gamma_{r,s}\cap\mathbb{A}\bigr]^*$ with  
$$\sum\limits_{i=1}^{r}{m_i}^2=\sum\limits_{j=1}^{s}{n_j}^2,$$ 
and let $d_0>0$ be the greatest common divisor of $(\mu,\,\nu)=(m_1,\cdots,m_r,n_1,\cdots,n_s)$. Define integers $m_i^{\prime}$ and $n_i^{\prime}$ through the 
equations $m_i={m_i}^{\prime}d_0$ and $n_j={n_j}^{\prime}d_0$. According to Proposition \ref{zero solution space} the solution space $\mathbb{M}({\bf n})=
\bigr\{{\bf\ell}\in\mathbb{Z}^{2N}~|~\Omega({\bf n})\: {{\bf\ell}}=0\bigr\}$ is given by:
\[
\mathbb{M}({\bf
  n})=\left\{
~\begin{pmatrix}
B(\nu^{\prime}) \ell \\-D(\mu^{\prime})\ell  \end{pmatrix}
\: 
~\Big{|}~{{\bf \ell}}=( \ell_1,\cdots,\ell_{N})^T\in\mathbb{Z}^{N}\right\},
\]
where $(\mu^{\prime},{\nu}^{\prime})=({m_1}^{\prime},\cdots,{m_r}^{\prime},{n_1}^{\prime},\cdots,{n_s}^{\prime})$. Hence 
\begin{equation}\label{trace2}
\text{{\bf tr}}\Big{(}e^{-t\mathcal{D}^{({\bf n})}}\Big{)}
=\frac{1}{(\pi \,t)^{N/2}}
\sum_{\ell \in \mathbb{Z}^{2N}}\limits\,
e^{-\frac{\Vert\mu\Vert^2\: \Vert{\bf\ell}\Vert^2}{{d_0}^{2}\,t}}
\left(\frac{2\Vert\mu\Vert}{\sinh\bigr(8\pi\,t\Vert\mu\Vert\bigr)}\right)^{N/2}.
\end{equation}
(3)\quad For ${\bf n}=2(\mu+\nu)$ with
$\Vert\mu\Vert\not=\Vert\nu\Vert$ 
the matrix $\Omega({\bf n})$ is non-singular.  
In this case the solution space $\mathbb{M}({\bf n})=\{0\}$ is trivial and
\begin{equation}\label{trace3}
\text{{\bf tr}}\Big{(} e^{-t\mathcal{D}^{({\bf n})}}\Big{)}=
{2^{N}}
\: \left(\frac{\Vert\mu\Vert^2-\Vert\nu\Vert^2}
{\sinh\{4\pi\,t(\Vert\mu\Vert+\Vert\nu\Vert)\}
\sinh\{4\pi\,t(\Vert\mu\Vert-\Vert \nu\Vert)\}}\right)^{N/2.}.
\end{equation}
\begin{rem}\label{remark_dimension_from_the_spectrum}
If $s=0$, then the matrix $\Omega({\bf n})$ is always non-singular for ${\bf n}\neq 0$. In this case we find: 
$$\text{{\bf tr}}\Big{(} e^{-t\mathcal{D}^{({\bf n})}}\Big{)}=\left(\frac{\|{\bf n}\|}{\sinh{2\pi t\|{\bf n}\|}}\right)^N.$$
{Furthermore, from the heat trace formula (\ref{GL_decomposition_of_operator_trace}) we conclude that the eigenvalues of the sub-Laplacian on 
$\Gamma_{r,0}(V)\backslash\mathbb{G}_{r,0}(V)$ are given by 
\begin{itemize}
\item $\lambda_l=2\pi^2\|l\|^2$ for $l\in\Z^{2N}$.
\item   $\beta_{n,m}=4\pi\|n\|(2m+N)$ for $n\in \Z^{r}\backslash\{0\}$ and $m\in\N$ with multiplicity $4^N\|n\|^N \binom{m+N-1}{N-1}.$
\end{itemize}
Since $\lambda_l^2\in \pi^4\Z$ and $\beta_{n,m}^2\in\pi^2\Z$, we can distinguish between these different numbers, i.e. knowing the eigenvalues, we can extract the dimension of the admissible module $V$. We conclude that if two manifolds $\Gamma_{r,0}(V)\backslash\mathbb{G}_{r,0}(V)$ and $\Gamma_{r^\prime,0}(V^\prime)\backslash\mathbb{G}_{r^\prime,0}(V^\prime)$ are isospectral with respect to the sub-Laplacian, then they have the same dimension.}

\end{rem}
Based on the dependence of the above heat traces on the parameters $\|\mu\|$, $\|\nu\|$ and $N$ we conclude: 
\begin{cor}\label{isospectrality}
$\Gamma_{r,s}\backslash \mathbb{G}_{r,s}$ and $\Gamma_{s,r}\backslash \mathbb{G}_{s,r}$ are isospectral 
with respect to the Laplacian and the sub-Laplacian, if the dimensions of their admissible modules coincide.
\end{cor}
The pseudo $H$-type algebra $\mathcal{N}_{r,s}$ does not depend on the chosen minimal admissible module, cf. \cite{FM}. Moreover, 
in the above determination we do not explicitly use the assumption that the admissible module is minimal. This implies:
\begin{thm}\label{multiple spectrum}
If $V$ is a sum of $k$ minimal admissible modules, 
then the heat trace in each of the cases, {\em(}\ref{trace1}{\em)}, {\em(}\ref{trace2}{\em)} 
and {\em(}\ref{trace3}{\em)} above is the $k$-th power of 
the corresponding heat trace for the manifold 
$\Gamma_{r,s}\backslash \mathbb{G}_{r,s}$.
Let $U$ be an admissible module of $C\ell_{s,r}$ with 
$\dim V=\dim U$, then the two nilmanifolds 
$\Gamma_{r,s}(V)\backslash\mathbb{G}_{r,s}(V)$ 
and $\Gamma_{s,r}(U)\backslash\mathbb{G}_{s,r}(U)$ are isospectral.
\end{thm}
\section{Isospectral, but non-diffeomorphic nilmanifolds}
\label{Section_non_diffeomorphic_isospectral}
By applying Theorem \ref{multiple spectrum} and the classification of pseudo $H$-type Lie algebras  in \cite{FM2,FM1} we detect finite families of 
{\it isospectral} but {\it mutually non-diffeomorphic} pseudo $H$-type nilmanifolds. If the module $V$ in the construction of the pseudo $H$-type Lie algebra 
is minimal admissible, then pairs of non-isomorphic pseudo $H$-type Lie algebras $\mathcal{N}_{r,s}\ncong \mathcal{N}_{s,r}$ of the 
same dimension $\dim \mathcal{N}_{r,s}= \dim \mathcal{N}_{s,r}$ are known (see \cite{FM2,FM1} or Table \ref{Table_Classification_PHLA}).  By choosing 
integral lattices $\Gamma_{r,s}$ and $\Gamma_{s,r}$ in the corresponding Lie groups $\mathbb{G}_{r,s}$ and $\mathbb{G}_{s,r}$, respectively, we first detect pairs 
\begin{equation}\label{GL_pairs_of_nilmanifolds}
M_1:=\Gamma_{r,s} \backslash \mathbb{G}_{r,s} \hspace{4ex} \mbox{\it  and} \hspace{4ex} M_2:=\Gamma_{s,r} \backslash \mathbb{G}_{s,r}
\end{equation}
of isospectral, non-diffeomorphic manifolds. The minimal dimension of such examples arise for $(r,s)=(1,3)$ in which case $\dim M_1= \dim M_2=12$. By dropping the minimality condition on 
the module we can produce many more examples.  More generally, for any $k \in \mathbb{N}$ we can find a family $M_1, \cdots, M_k$ of nilmanifolds such that for $i,j \in \{1, \cdots, k\}$:
\begin{equation}\label{GL_family_of_nilmanifolds}
M_i \: \sim_{\textup{isosp}} \: M_j \hspace{3ex} \mbox{\it and } \hspace{3ex} M_i \nsim_{\textup{diffeo}} M_j.
\end{equation} 
Here $\sim_{\textup{isosp}}$ means {\it isospectral} with respect to the sub-Laplacian and $\nsim_{\textup{diffeo}}$ indicates that two manifolds are {\it non-diffeomorphic}.  First, we explain the 
method of detecting non-diffeomorphic nilmanifolds. 
\vspace{1ex}\par 
If there is a diffeomorphism between the nilmanifolds $N_{r,s}:=\Gamma_{r,s}(V)\backslash\mathbb{G}_{r,s}(V)$ and $N_{s,r}=\Gamma_{s,r}(U)
\backslash\mathbb{G}_{s,r}(U)$, then their fundamental  groups
$$\pi_1\left(\Gamma_{r,s}(V)\backslash\mathbb{G}_{r,s}(V)\right)\cong \Gamma_{r,s}(V) 
\hspace{3ex} \mbox{and} \hspace{3ex} 
\pi_1\left(\Gamma_{s,r}(U)\backslash\mathbb{G}_{s,r}(U)\right)\cong \Gamma_{s,r}(U)$$ 
are isomorphic. We can apply the following general fact from \cite{Ra}:
\begin{prop}\label{Proposition_Raghunathan}
Any isomorphism between lattices in simply connected nilpotent Lie groups can be extended to an isomorphism between the whole groups. 
\end{prop}
From these observations we conclude: 
\begin{cor}\label{corollary_non_diffeomorphic_nilmanifolds}
If the nilmanifolds $N_{r,s}$ and $N_{s,r}$ are diffeomorphic, then $\mathcal{N}_{r,s}(V)$ and $\mathcal{N}_{s,r}(U)$ have to be isomorphic as Lie algebras. 
\end{cor}
\subsection{Pairs of non-diffeomorphic, isospectral nilmanifolds via minimal admissible modules}
\label{subsection?7-1}
To obtain pairs $(M_1=N_{r,s}, M_2=N_{s,r})$ of nilmanifolds with (\ref{GL_family_of_nilmanifolds}) we determine pairs $(r,s)$ such that $\text{dim}\:N_{r,s}=\text{dim}\:N_{s,r}$, but 
$N_1\ncong N_2$. The classification of pseudo $H$-type algebras in {\sc Table \ref{t:step1}} constructed from minimal admissible modules was obtained in \cite{FM2,FM1}. 
\begin{table}[]
\label{Table_Classification_PHLA}
\center
\caption{{\it Classification of pseudo $H$-type Lie algebras defined via minimal admissible modules.}
{\bf Notation}: If $V_{\textup{min}}^{r,s}$ denotes a minimal admissible $C\ell_{r,s}$-module,  then the letter `d' = {\it double} (or `h'={\it half}, respectively) at the position $(r,s)$ means 
that $\dim V_{\textup{min}}^{r,s}= 2 \dim V_{\textup{min}}^{s,r}$ (or $\dim V_{\textup{min}}^{r,s}= 1/2 \dim V_{\textup{min}}^{s,r}$, respectively).  The symbol `$\cong$' indicates that the 
Lie algebras $\mathcal{N}_{r,s}$ and $\mathcal{N}_{s,r}$ are isomorphic while `$\not \cong$' means that they are non-isomorphic.}
\begin{tabular}{|c||c|c|c|c|c|c|c|c|c|}
\hline
8&$\cong$&$\cong$&$\cong$&h&&&&&
\\
\hline
7&d&d&d&$\not\cong$&&&&&
\\
\hline
6&d&$\cong$&$\cong$&h&${\cong}$&&&&
\\
\hline
5&d&$\cong$&$\cong$&h&${\cong}$&$ $&&&
\\
\hline
4&$\cong$&h&h&h&$ $&${\cong}$&${\cong}$&&\\\hline
3&d&$\not\cong$&$\not\cong$&$ $&d&d&d&$\not\cong$&d\\\hline
2&$\cong$&h&$ $&$\not\cong$&d&$\cong$&$\cong$&h&$\cong$\\\hline
1&$\cong$&$ $&d&$\not\cong$&d&$\cong$&$\cong$&h&$\cong$\\\hline
0&&$\cong$&$\cong$&h&$\cong$&h&h&h&$\cong$\\\hline\hline
$s/r$&0&1&2&3&4&5&6&7&8\\
\hline
\end{tabular}\label{t:step1}
\end{table}
This table gives us only information about the cases $0\leq r,s\leq 8$. For the remaining cases we use the following periodicity:  
\begin{lem}\label{Periodicity_pseudo_H_type_lemma}
For $(\mu,\nu)\in \{(8,0),(0,8)\}\text{ mod }8$ and $(\mu,\nu)=(4,4)\text{ mod }4$ we have
$$N_{r,s}\cong N_{s,r}\text{ iff } N_{r+\mu,s+\nu}\cong N_{s+\nu,r+\mu}.$$
\end{lem}
From Corollary \ref{corollary_non_diffeomorphic_nilmanifolds} and {\sc Table \ref{t:step1}} we see that for $(r,s)\in\{(3,1),(3,2),(3,7)\}$ both nilmanifolds $N_{r,s},N_{s,r}$ 
have the same dimension but they are non-diffeomorphic. 
\vspace{1mm}\par 
Hence we obtain the following result:
\begin{cor}\label{Corollary_non-diffeomorphic_isospectral_first_method}
The following pairs of nilmanifolds are isospectral and non-diffeomorphic:
\begin{enumerate}
\item $\left(N_{r,s},N_{s,r}\right)$ for $r\equiv 3 \: \textup{mod}\: 8$ and $s\equiv 1,2,7 \: \textup{mod} \: 8$.
\item $\left(N_{r+4k,s+4k},N_{s+4k,r+4k}\right)$ for $(r,s)\in\{(3,1),(3,2),(3,7)\}$ and $k\in\mathbb{N}_0.$
\end{enumerate}
\end{cor}
\subsection{Finite families of non-diffeomorphic, isospectral nilmanifolds}
\label{subsection?7-2}
In case the module $V$ in the construction of the Lie algebra $\mathcal{N}_{r,s}(V)$ is not minimal admissible we can use the classification result in 
\cite[Theorem 4.1.2 and Theorem 4.1.3]{FM2} to determine families $\{M_1, \cdots, M_k\}$ of a given lenght $k \in \mathbb{N}$ of isospectral, mutually non-diffeomorphic nilmanifolds, i.e. (\ref{GL_family_of_nilmanifolds}) holds. 
First, we fix the pair $(r,s)$ and study the Lie algebra $\mathcal{N}_{r,s}(U)$, constructed from different admissible modules. In the general case the classification of isomorphic pseudo $H$-type 
Lie algebras is more subtle and to state the result we need to introduce some notation from \cite{FM2}. Note that for any given minimal admissible module $\{ J,V, \langle \bullet, \bullet \rangle_V \}$ 
also the module $\{ J,V, -\langle \bullet, \bullet \rangle_V \}$ is minimal admissible. The upper index in the notation $V_{\textup{min};\pm}^{r,s;\pm}$ indicate that the scalar product 
of the two minimal admissible modules  $V_{\textup{min};\pm}^{r,s;+}$ and  $V_{\textup{min};\pm}^{r,s;-}$ differ by a sign. 
\begin{thm}[{K. Furutani, I. Markina}, \cite{FM2}]\label{Theorem_classification_II_1_case}
Let $r\equiv 3 \: \textup{mod} \: 4$, $s\equiv 1,2,3 \: \textup{mod}\: 4$ and $U,\widetilde{U}$ be admissible modules decomposed into the direct sums: 
$$U=\left(\bigoplus^{p^+} V_{\textup{min}}^{r,s;+}\right)\bigoplus\left(\bigoplus^{p^-}V_{\textup{min}}^{r,s;-}\right),$$
$$\widetilde{U}=\left(\bigoplus^{\widetilde{p}^+} V_{\textup{min}}^{r,s;+}\right)\bigoplus\left(\bigoplus^{\widetilde{p}^-}V_{\textup{min}}^{r,s;-}\right).$$
Then the Lie algebras $\mathcal{N}_{r,s}(U), \mathcal{N}_{r,s}(\widetilde{U})$ are isomorphic, if and only if:
$$\Big{[}p^+=\widetilde{p}^+ \hspace{1ex} \text{ and }\hspace{1ex}  p^-=\widetilde{p}^- \Big{]} \hspace{3ex} \mbox{or} \hspace{3ex}  
\Big{[}p^+=\widetilde{p}^- \hspace{1ex} \text{ and }\hspace{1ex} p^-=\widetilde{p}^+\Big{]}.$$
\end{thm}
Let $R \in \mathbb{N}$ be fixed. If we consider admissible modules of the form 
 $$U(p,q):=\left(\bigoplus^{p} V_{\textup{min}}^{r,s;+}\right)\bigoplus\left(\bigoplus^{q}V_{\textup{min}}^{r,s;-}\right), \hspace{2ex} p+q=R, $$
 then we obtain non-isomorphic Lie algebras $\mathcal{N}_{r,s}(U(p_1,q_1))$ and $\mathcal{N}_{r,s}(U(p_2,q_2))$ of the same dimension 
 $R\cdot \dim V_{min}^{r,s;+}$ if simultaneously $(p_1,q_1) \ne (p_2,q_2)$ and $(p_1,q_1) \ne (q_2, p_2)$. 
 \vspace{1mm} \par 
 We fix an integer $k$ and determine all pairs of integers $(p_i,q_i)$ with the properties: 
 \begin{enumerate}
 \item $p_i\leq q_i$ for all $i.$
 \item  $p_i+q_i=k$ for all $i.$
 \item $(p_i,q_i)\neq (p_j,q_j)$ and $(p_i,q_i)\neq (q_j,p_j)$ for $i\neq j.$
 \end{enumerate}
With such pairs we define: 
 $$U_i:=\left(\bigoplus^{p_i} V_{\textup{min}}^{r,s;+}\right)\bigoplus\left(\bigoplus^{q_i}V_{\textup{min}}^{r,s;-}\right).$$
 \par 
From Theorem \ref{Theorem_classification_II_1_case} and the above remark we conclude that the Lie algebras $\mathcal{N}_{r,s}(U_i)$ and $\mathcal{N}_{r,s}(U_j)$ are 
mutually non-isomorphic for $r\equiv 3\:  \textup{mod} \:4$, $s\equiv 1,2,3 \: \textup{mod}\: 4$. In order to present a concrete family of nilmanifolds with the required properties 
we choose $r=3$ and $s=1$ such that $\text{dim}\:V_{\textup{min}}^{3,1}=8$. Let $k$ be even and choose $m\in \mathbb{N}$ such that $k=2m$. 
 The pairs $(p_i,q_i)$ with (a) - ({c}) are of the form $\{(i,k-i) \: | \: 0\leq i\leq m\}$ and we conclude:
 \begin{cor}
 The following $m+1$ nilmanifolds 
 $$\big{(}\Gamma_{3,1}\backslash\mathbb{G}_{3,1}(U_i) \big{)}_{0\leq i\leq m}$$ 
 are isospectral, but mutually non-diffeomorphic with respect to the (sub)-Laplacian. 
 \end{cor}
 \begin{rem}
 For any given integer $m\in\mathbb{N}$ and by using the above method, we can construct $m+1$ nilmanifolds of the common dimension $4+16m$ which are isospectral but mutually non-diffeomorphic.  
 In particular, one obtains a pair of such manifolds of the (minimal) dimension $4+16\times 1=20$. Note that via the first method (i.e. Corollary \ref{Corollary_non-diffeomorphic_isospectral_first_method}) 
we can find a pair of such nilmanifolds of dimension $12$. 
 \end{rem}
 To minimize the dimension of the constructed family of nilmanifolds we should use a third method which is based on \cite[Theorem 4.2]{FM2}, which treat the case $r\equiv 3 \text{ mod } 4$ and 
 $s\equiv 0 \text{ mod} 4$. In this situation there are two non-equivalent irreducible modules and we use the lower index $\pm$ in the 
 notation below to distinguish the minimal admissible modules corresponding to each irreducible modules (or to each sum of irreducible modules , cf. Remark \ref{Remark_admissible_modules_5_cases}).
 \begin{thm}[{K. Furutani, I. Markina}, \cite{FM2}]\label{Theorem_classification_II_2_case}
 For $r\equiv 3\: \textup{mod} \: 4$ and $s\equiv 0 \: \textup{mod} \: 4$, let $U,\widetilde{U}$ be admissible modules decomposed into the direct sums:
 $$U=\left(\bigoplus^{p^+_+} V_{\textup{min},+}^{r,s;+}\right)\bigoplus\left(\bigoplus^{p^-_+}V_{\textup{min},+}^{r,s;-}\right)\bigoplus 
 \left(\bigoplus^{p^+_-}V_{\textup{min},-}^{r,s;+}\right)\bigoplus \left(\bigoplus^{p^-_-}V_{\textup{min},-}^{r,s;-}\right),$$
  $$\widetilde{U}=\left(\bigoplus^{\widetilde{p}^+_+} V_{\textup{min},+}^{r,s;+}\right)\bigoplus\left(\bigoplus^{\widetilde{p}^-_+}V_{\textup{min},+}^{r,s;-}\right)\bigoplus 
  \left(\bigoplus^{\widetilde{p}^+_-}V_{\textup{min},-}^{r,s;+}\right)\bigoplus \left(\bigoplus^{\widetilde{p}^-_-}V_{\textup{min},-}^{r,s;-}\right).$$
  Then the Lie algebras $\mathcal{N}_{r,s}(U)$ and  $\mathcal{N}_{r,s}(\widetilde{U})$ are isomorphic, if and only if one of the following conditions are fulfilled: 
 $$\Big{[}p^+_++p^-_-=\widetilde{p}^+_++\widetilde{p}^-_-\hspace{1ex} \text{\it and}\hspace{1ex} p^-_++p^+_-=\widetilde{p}^-_++\widetilde{p}^+_-\Big{]},$$
 or
 $$\Big{[} p^+_++p^-_-=\widetilde{p}^-_++\widetilde{p}^+_-\hspace{1ex} \text{\it and} \hspace{1ex} p^-_++p^+_-=\widetilde{p}^+_++\widetilde{p}^-_-\Big{]}. $$
 \end{thm}
 To simplify the construction we choose $p^+_-=p^-_-=0$, $\widetilde{p}^+_-=\widetilde{p}^-_-=0$. The condition in Theorem \ref{Theorem_classification_II_2_case} take the form: 
 $$\Big{[}p^+_+=\widetilde{p}^+_+\hspace{1ex} \text{\it and}\hspace{1ex} p^-_+=\widetilde{p}^-_+\Big{]} \hspace{3ex} \mbox{or} \hspace{3ex} 
 \Big{[} p^+_+=\widetilde{p}^-_+\hspace{1ex} \text{\it and}\hspace{1ex} p^-_+=\widetilde{p}^+_+\Big{]}.$$
 Next, we choose  $r=3$, $s=0$ and fix $m\in \mathbb{N}$. Then we consider the following family of $m+1$ admissible modules:
 $$V_i:=\left(\bigoplus^{i} V_{\textup{min};+}^{3,0;+}\right)\bigoplus\left(\bigoplus^{2m-i}V_{\textup{min};+}^{3,0;-}\right), \hspace{3ex} (0 \leq i \leq m).$$
We obtain a family of $m+1$ mutually non-diffeomorphic, isospectral nilmanifolds of common dimension $3+8m$.
 \begin{cor}
 For $0\leq i\leq m$, the $m+1$ nilmanifolds 
 $$\left(\Gamma_{3,0}\backslash\mathbb{G}_{3,0}(V_i) \right)_{0\leq i\leq m}$$
 are isospectral but mutually non-diffeomorphic.
 \end{cor}
 \begin{rem}
 By choosing $m=1$ we obtain a pair of nilmanifolds both having dimension $3+8=11$. This dimension is minimal among the previous examples. 
 \end{rem} 
\section{Subriemannian structure and heat trace expansion}
 To every pseudo $H$-type nilmanifold $M=\Gamma_{r,s}\backslash\mathbb{G}_{r,s}$ with $r+s>1$ and based on \cite[Theorem 3.3]{BFI1} we construct a Heisenberg manifold 
 $H=\Gamma \backslash \mathbb{H}_{2n+1}$ such that the short time heat trace asymptotic expansions corresponding to the sub-Laplacians on $M$ and $H$, 
respectively, coincide  up to a term vanishing to infinite order. Moreover, in our construction the manifolds $M$ and $H_{\alpha}$ have different dimensions. Recall that in the 
case of a Riemannian structure on $M$  the heat trace expansion corresponding to the Laplacian encodes the dimension of $M$ and therefore such examples do not 
exist in the framework of Riemannian geometry.
\vspace{1ex} \par 
Let  $d=r+s>1$, and with our previous notation consider the nilmanifold 
$$M=\Gamma_{r,s}\backslash\mathbb{G}_{r,s} \hspace{2ex} \text{ \it with } \hspace{2ex}  \text{dim } M=2N+d. $$
We write $\mathbb{H}_{2n+1}=\mathbb{G}_{1,0}(\R^{2n})$ for the $(2n+1)$-dimensional Heisenberg group and with $\alpha >0$ we define a lattice $\Gamma_{\alpha} \subset \mathbb{H}_{2n+1}$ 
of the form:
$$\Gamma_\alpha=\left\{\sqrt{\alpha}\sum m_i X_i+\frac{\alpha}{2} k Z:m_i, k\in\Z\right\}.$$
Here $\{X_i,Z\; | \; i=1, \cdots, 2n\}$ denotes a basis of the Lie algebra of $\mathbb{H}_{2n+1}$ with the non-trivial bracket relations 
\begin{equation*}
\big{[}X_i, X_{n+j} \big{]}=\delta_{ij} Z, \hspace{6ex} (i,j=1, \cdots,n). 
\end{equation*}
The corresponding one-parameter family of Heisenberg manifolds will be denoted by: 
$$H_\alpha=\Gamma_\alpha\backslash \mathbb{H}_{2n+1}\hspace{3ex} \mbox{\it where } \hspace{3ex} \alpha >0.$$
We recall the form of the short time heat trace asymptotic expansion of the sub-Laplacian on a general compact 2-step nilmanifold $\Gamma\backslash\mathbb{G}$ in \cite[Theorem 3.3]{BFI1}: 
\begin{thm}[{W. Bauer, K. Furutani, C. Iwasaki \cite{BFI1}}] \label{Theorem_heat_trace_expansion_Bauer_Furutani_Iwasaki}
Let $M=\Gamma\backslash\mathbb{G}$  be a $2$-step compact nilmanifold of dimension $2N+d$. Then
$${\bf tr} \left(e^{-t \Delta_{\textup{sub}}^M}\right)=\frac{c_{M}}{t^{N+d}}+O(t^\infty) \hspace{3ex} \mbox{\it as} \hspace{3ex} t \rightarrow 0.$$
The constant $c_{M} >0$ explicitly is given by 
\begin{equation}
\label{GL_coefficients_heat_trace_expansion}
c_{M}=\frac{\textup{Vol}(M)}{(2\pi)^{N+d}}\int_{\R^d}W(\tau)d\tau.
\end{equation}
\end{thm}
If we apply the above theorem to $H_\alpha$, we obtain: 
 $$c_{H_\alpha}=\frac{\textup{Vol}(H_\alpha)}{(2\pi)^{n+1}}\int_{\R}W(\tau)d\tau.$$
 Here $W(\tau)d \tau$ is simply the volume form in Theorem \ref{heat kernel by BGG1} and associated to the Heisenberg group with the above structure constants. Note that 
 $$\textup{Vol}(H_\alpha)=\alpha^{n+1}\textup{Vol}(H_1).$$
 If we choose $n\in \mathbb{N}$ with $n+1=N+d$ and $\alpha>0$ and such that $c_{H_\alpha}=c_M$, then we obtain a pair of compact, subriemannian manifolds with the properties:
 \begin{enumerate}
 \item ${\bf tr}\left(e^{-t \Delta_{\textup{sub}}^M}\right)-{\bf tr}\left(e^{-t \Delta_{\textup{sub}}^{H_\alpha}}\right)=O(t^\infty)\text{ as }t\to 0$\\
 \item $2N+d=\text{dim }M\neq \text{dim }H_\alpha=2N+2d-1$ (since $d>1$). 
 \end{enumerate}\par 
 {
  \begin{rem}
 From the heat trace expansion for small times we can read the Hausdorff dimension $2(N+d)$ of the nilmanifold 
 $M=\Gamma\backslash\mathbb{G}$ in Theorem \ref{Theorem_heat_trace_expansion_Bauer_Furutani_Iwasaki} considered as a metric space with respect to the  Carnot-Carath\'{e}odory distance. However, the last example indicates that we cannot read the Euclidean dimension of $M$ from the  coefficient (\ref{GL_coefficients_heat_trace_expansion}) of the heat trace expansion. However, in Remark \ref{GL_decomposition_of_operator_trace} we have pointed out that in some cases this dimension can be obtained from the full spectrum of the 
 sub-Laplacian. In the case of the nilmanifolds $M$ in this paper and which are constructed from a standard lattice we have $\textup{vol}(M)=1$ and therefore, with the notation in Theorem \ref{Theorem_heat_trace_expansion_Bauer_Furutani_Iwasaki}: 
 \begin{equation}\label{GL_coefficients_heat_trace_expansion_standard_lattice}
 c_M=C_M(N,d):= \frac{1}{(2\pi)^{N+d}} \int_{\mathbb{R}^d} W(\tau) d\tau. 
 \end{equation}
 \end{rem}
 We list a few problem concerning the geometric information contained in the spectrum of the sub-Laplacian on a nilmanifold: 
 \vspace{1mm}\\
 {\bf Problems}: 
 \begin{itemize}
 \item[(a)] Let $\mathbb{G}_{r,s}(V)$ be a pseudo $H$-type group with standard lattice $\Gamma_{r,s}$ as explained in Section \ref{Section_Pseudo_H_type_algebras}. 
 Can we determine the numbers $2N= \dim V$ and $d= \dim \mathbb{G}_{r,s}(V)-2N$ from the coefficient (\ref{GL_coefficients_heat_trace_expansion_standard_lattice})? 
\vspace{1mm}\\
Consider the case $s=0$. Then integration with respect to polar coordinates shows: 
\begin{align*}
 \int_{\mathbb{R}^d} W(\tau) d\tau=\int_{\mathbb{R}^d} \frac{|\tau|^N}{(\sinh |\tau|)^N} d\tau = 2 V_d \int_0^{\infty} \frac{r^{N+d-1}e^{-Nr}}{(1-e^{-2r})^N} dr=(*), 
\end{align*}
where $V_d= 2 \pi^{\frac{d}{2}}/ \Gamma(\frac{d}{2})$ denotes the volume of the $(d-1)$-dimensional unit sphere. Now we use the following power series expansion for $|x|<1$: 
\begin{equation*}
\frac{1}{(1-x)^N}= \sum_{\alpha \in \mathbb{N}_0^N}x^{|\alpha|}. 
\end{equation*}
A change of variables in the integral is applied to obtain: 
\begin{align*}
(*)&=2V_d \sum_{\alpha \in \mathbb{N}_0^N} \int_0^{\infty} r^{N+d-1} e^{-(N+2|\alpha|) r} dr\\
&= 2 V_d \sum_{\alpha \in \mathbb{N}_0^N} \frac{1}{(N+2|\alpha|)^{N+d}} \int_0^{\infty} r^{N+d-1} e^{-r} dr\\
&=  \frac{V_d \Gamma(N+d) }{2^{N+d-1}}\sum_{\alpha \in \mathbb{N}_0^N} \frac{1}{\big{(}\frac{N}{2}+|\alpha|\big{)}^{N+d}}\\
&=   \frac{V_d \Gamma(N+d) }{2^{N+d-1}}\zeta_n\big{(}N+d, \frac{N}{2} \big{)}. 
\end{align*}
The infinite sum is called multiple Hurwitz zeta function and previously has been studied in the literature: 
\begin{equation*}
\zeta_n\big{(} N+d, \frac{N}{2} \big{)}:=\sum_{\alpha \in \mathbb{N}_0^N} \frac{1}{\big{(}\frac{N}{2}+|\alpha|\big{)}^{N+d}}. 
\end{equation*}
Hence: 
\begin{equation*}
c_M(N,d) (2\pi)^{N+d}  \frac{2^{N+d-2}}{\Gamma(N+d)} = \frac{\pi^{\frac{d}{2}}}{\Gamma\big{(} \frac{d}{2} \big{)}} \zeta_n\Big{(}N+d, \frac{N}{2} \Big{)}. 
\end{equation*}
The left hand side can be calculated from the spectral data (more precisely, from the heat trace expansion in Theorem \ref{Theorem_heat_trace_expansion_Bauer_Furutani_Iwasaki}). 
{Hence the problem reduces to the question, whether for each $k \in \mathbb{N}$ the assignment: 
\begin{equation*}
N_k:=\Big{\{} (N,d)\in \mathbb{N}^2 \: : \: N+d=k \Big{\}} \ni (N,d)  \mapsto  \frac{\pi^{\frac{d}{2}}}{\Gamma\big{(} \frac{d}{2} \big{)}} \zeta_n\Big{(}k, \frac{N}{2} \Big{)}
\end{equation*}
is injective. }
\item[(b)] Consider wo isospectral compact nilmanifolds $M_j=\Gamma_j\backslash G_j$  where $j=1,2$. Assume that both are  equipped with a left-invariant subriemannian structure as described in this paper. Is it true that $\dim G_1= \dim G_2$ (see Remark \ref{remark_dimension_from_the_spectrum})?
\item[(c)] Let $M:=\Gamma \backslash G$ denote a compact nilmanifold (e.g. modelled over a pseudo $H$-type Lie group). Determine the 
asymptotic  growth of the eigenvalue counting function for the corresponding sub-Laplacian. Based on a classification of lattices and the explicit spectral data such question in the case of Heisenberg manifolds has been discussed in \cite{Str}. 
 \end{itemize}}
\section{Appendix}
In the appendix we present the dimensions of minimal admissible modules for some basic cases in form of a table. These data are taken from  \cite{FM2,FM1} which we refer to for more details and notations. The remaining cases can be obtained by $(4,4)$, $(8,0)$ and $(0,8)$-periodicities with respect
to the signature $(r,s)$, 
respectively. 
In particular, the table indicates the cases in which two non-equivalent minimal admissible modules exist. 
However, it is known that  pseudo $H$-type algebras constructed from two non-equivalent minimal admissible modules are isomorphic.
\begin{table}[h]
	\center
\caption{Dimensions of minimal admissible modules}
\begin{tabular}{|c||c|c|c||c||c|c|c||c||c|}
\hline
${\small 8} $&$ {16}^{\pm}$&$$&$$&$$&$$&$$&$$&$$&$$
\\
\hline\hline
${\small 7}$&$ {16}^{N}$&${32}^{N}$&${{\bf 64}^{N}} $&${\small 64}^{\pm}$&$$&$$&$$&$$&$$
\\
\hline
${\small 6}$ &${16}^{N}$&${16}_{\times 2}^{N}$&${32 }^{N}$&${32}^{\pm}$&$$&$$&$$&$$&$$
\\
\hline
${\small 5} $&${{\bf 16}^{N}}$&${16}^{N}$&${16}^{N}$&${\small 16}^{\pm}$&$$&$$&$$&$$&$$
\\
\hline\hline
${\small 4} $&${8}^{\pm}$&${8}^{\pm}$&${8}^{\pm}$&$ {8}_{\times 2}^{\pm}$&${16}^{\pm}$&$$&$$&$$&$$
\\
\hline\hline
${\small 3}$&${{\bf 8}^{N}}$&${{8}^{N}}$&${{8}^{N}}$&${\small 8}^{\pm}$&${16}^{N}$&${32}^{N}$&${{\bf 64}^{N}}$&${64}^{\pm}$&$$
\\
\hline
${\small 2}$&${{\bf 4}^{N}}$&${{\bf 4}_{\times 2}^{N}}$&${{\bf 8}^{N}}$&$ 8^{\pm}$&${16}^{N}$&${16}_{\times 2}^{N}$&${32}^{N}$&${32}^{\pm} $&$$
\\
\hline
${\small 1}$ &${{\bf 2}^{N}}$&${{\bf 4}^{N}}$&${{\bf 8}^{N}}$& ${\small 8}^{\pm}$&${{\bf 16}^{N}}$&${16}^{N}$&${16}^{N}$&${16}^{\pm}$&$$
\\
\hline\hline
${\small 0}$&${1}^{\pm}$&${2}^{\pm}$&${4}^{\pm}$&${4}_{\times 2}^{\pm}$&$ {8}^{\pm}$&${8}^{\pm}$&${8}^{\pm}$&${8}_{\times 2}^{\pm}$&${16}^{\pm}$
\\
\hline\hline
$\text{s/r}$&${0}$&${1}$&${2}$&${3}$&${4}$&${5}$&${6}$&${7}$&${8}$
\\
\hline
\end{tabular}\label{t:dim}
\end{table}
\begin{center}
black = irreducible,\quad
{\bf bold} = double of irreducible, \\
$*_{\times 2}=$ two non-equivalent minimal dimensional admissible modules
\end{center}

\end{document}